\documentclass[notes-porbanz,noinfoline]{notes-imsart}
\pdfoutput=1

\RequirePackage[OT1]{fontenc}
\RequirePackage{amsthm,amsmath,amssymb}
\usepackage[usenames,dvipsnames]{color}
\RequirePackage[colorlinks,linkcolor=blue,citecolor=blue]{hyperref}
\setattribute{copyright}{text}{\empty}
\hypersetup{pdfinfo={Subject={ }}}
\usepackage[numbers]{natbib}
\usepackage[capitalize]{cleveref}

\startlocaldefs
\numberwithin{equation}{section}
\theoremstyle{plain}
\pdfoutput=1

\theoremstyle{plain}

\newtheorem{theorem}{Theorem}[section]
\newtheorem{lemma}[theorem]{Lemma}
\newtheorem{proposition}[theorem]{Proposition}
\newtheorem{corollary}[theorem]{Corollary}
\theoremstyle{definition}

\newtheorem{remark}[theorem]{Remark}

\newcommand{\kword}[1]{{\bf #1}\index{keywords}{#1}}

\def\ie{i.e.\ }
\def\eg{e.g.\ }
\def\mean#1{\mathbb{E}[#1]}

\def\Bigmean#1{\mathbb{E}\Bigl[#1\Bigr]}

\def\equdist{\stackrel{\mbox{\tiny d}}{=}}
\def\iid{i.i.d.\ }
\def\condind{{\perp\!\!\!\perp}}
\def\LAW{\mathcal{L}}

\def\abstmeasure{\mathbb{P}}

\def\xspace{\mathcal{X}}
\def\yspace{\mathcal{Y}}

\def\kword#1{\textbf{#1}}
\def\Z{\mathbf{Z}}

\def\U{\mathbf{U}}

\def\atomspace{\mathbb{U}}
\def\argdot{\bullet}
\def\equdist{\stackrel{\mbox{\tiny\rm d}}{=}}
\def\BP#1{\text{\rm BeP}(#1)}
\def\Bernoulli{\mbox{Bernoulli}}

\def\simiid{\sim_{\mbox{\tiny\rm iid}}}
\def\Law{\mathcal{L}}
\def\mass#1{T_{#1}}
\def\PK{\mbox{\rm PK}}
\def\Beta{\text{\rm Beta}}
\def\Poisson{\text{\rm Poisson}}
\def\GammaDist{\text{\rm Gamma}}
\def\Stable{\text{\rm Stable}}
\def\Uniform{\text{\rm Uniform}}
\def\indicator#1{\mathbb{I}\lbrace #1 \rbrace}
\def\Levy{L\'evy }
\def\PD{\text{\rm PD}}

\def\DeltaAlt{\Delta^{\!\circ}}
\def\xiAlt{\xi^{\circ}}
\def\xiAltAlt{\eta}
\def\TAlt{T^{\circ}}
\def\PAlt{P^{\circ}}
\def\QAlt{Q^{\circ}}
\def\PK{\text{\rm PK}}
\def\IBP{\text{\rm IBP}}
\def\BeP{\text{\rm BeP}}
\def\CRP{\text{\rm CRP}}

\def\JOT{\text{\rm GD}} 

\def\IBP{\text{\rm IBP}}
\def\dtv{d_{\text{\tiny TV}}}
\def\ind#1{\mbox{\tiny #1}}
\def\RS{\Sigma^{\ind{R}}}

\def\Ma{M^{\ind{\rm ($a$)}}}

\usepackage{booktabs}
\usepackage{tikz}
\usetikzlibrary{matrix,arrows,shapes.geometric,backgrounds,fit,chains,positioning,scopes}
\endlocaldefs

\begin{document}

\begin{frontmatter}
  \title{Scaled subordinators and generalizations of the Indian buffet process}
  \runtitle{Scaled subordinators}
  \begin{aug}
    \author{\fnms{Lancelot F.\ }\snm{James}\ead[label=e1]{lancelot@ust.hk}},
    \author{\fnms{Peter\ }\snm{Orbanz}\corref{}\ead[label=e2]{porbanz@stat.columbia.edu}}
    \and
    \author{\fnms{Yee Whye\ }\snm{Teh}%
      \ead[label=e3]{y.w.teh@stats.ox.ac.uk}}%
    \runauthor{James, Orbanz and Teh}
    \affiliation{HKUST, Columbia University and University of Oxford}
    \address{Deparment of Information Systems, Business\\ Statistics, and Operations Management\\
      Clear Water Bay, Kowloon\\
      Hong Kong\\
      \printead{e1}
    }
    \address{Department of Statistics\\
      1255 Amsterdam Avenue\\
      New York, NY-10027, USA\\
      \printead{e2}
    }
    \address{Department of Statistics\\
      1 South Parks Road\\
      Oxford OX1 3TG, U.K.\\
      \printead{e3}
    }    
  \end{aug}
  \maketitle
  \begin{abstract}
    We study random families of subsets of $\mathbb{N}$ that are similar to exchangeable random partitions, but do not require constituent sets to be disjoint: Each element of ${\mathbb{N}}$ may be contained in multiple subsets. One class of such objects, known as Indian buffet processes, has become a popular tool in machine learning. Based on an equivalence between Indian buffet and scale-invariant Poisson processes, we identify a random scaling variable whose role is similar to that played in exchangeable partition models by the total mass of a random measure. Analogous to the construction of exchangeable partitions from normalized subordinators, random families of sets can be constructed from randomly scaled subordinators. Coupling to a heavy-tailed scaling variable induces a power law on the number of sets containing the first $n$ elements. Several examples, with properties desirable in applications, are derived explicitly. A relationship to exchangeable partitions is made precise as a correspondence between scaled subordinators and Poisson-Kingman measures, generalizing a result of Arratia, Barbour and Tavar\'e on scale-invariant processes.
  \end{abstract}
  
\end{frontmatter}

\def\rho{\lambda}

\section{Overview and main results}
\label{sec:introduction}

The Indian buffet process, or IBP, of \citet{Griffiths:Ghahramani:2006} is a 
distribution on families of sets, encoded as binary matrices.
This model and its two- and three-parameter generalizations 
\citep{Thibeaux:Jordan:2007,Teh:Goerur:2009} have received considerable attention in
machine learning \citep{Ghahramani:Griffiths:Sollich:2006:1,
  Zhou:Hannah:Dunson:Carin:2012:1,
  Titsias:2008:1,
  Paisley:Zaas:Ginsburg:Woods:Carin:2010,
  Broderick:Jordan:Pitman:2012:1,
  Heaukulani:Roy:2015:1},
where applications include dyadic data \citep{Meeds:Ghahramani:Neal:Roweis:2007},
link prediction \citep{miller2009nonparametric},
time series \citep{Fox:AOAS2014},
user preference data \citep{navarro2008latent},
and networks \citep{Caron:2012:1}.
Theoretical results include the work of \citet*{Broderick:Jordan:Pitman:2013:B}, who
generalize exchangeable partition probability functions to families of sets,
and of \citet*{Berti:Crimaldi:Pratelli:Rigo:2014:1}, who establish central limit theorems.

Random families of sets are related to random partitions of $\mathbb{N}$---the former differ from
the latter in that each ${n\in\mathbb{N}}$ may be contained in multiple blocks, or in no block at all.
The law of an exchangeable random partition can be
parametrized by a random probability measure, via Kingman's representation \citep[see][]{Pitman:2006}.
Similarly, the laws of a large class of random families
of sets can be parametrized by random measures, though these are not normalized
\citep{Thibeaux:Jordan:2007}.

The random measure defining the one-parameter IBP of \citet{Griffiths:Ghahramani:2006} 
derives from the scale-invariant Poisson process \citep[e.g.][]{Arratia:1998:1,Penrose:Wade:2004:1}.
We use this process to relate the IBP to the one-parameter Chinese restaurant process, 
via a remarkable result of \citet*{Arratia:Barbour:Tavare:1999:1}.
Beyond the scale-invariant case, a random scaling variable arises naturally, 
and we show that Poisson-Kingman partitions \citep{Pitman:2003:1} correspond
to a class of random set families defined by randomly scaled subordinators. For the
IBP and CRP, this result specializes to that of \citet{Arratia:Barbour:Tavare:1999:1}.
The random measures so defined have a number of generic properties, including
a ``stick-breaking'' representation. Substituting a stable for the scale-invariant
subordinator, we obtain explicit examples, some with properties reminiscent
of the two-parameter Poisson-Dirichlet \citep{Pitman:Yor:1997}. 
If the scaling variable is chosen heavy-tailed, the number of subsets containing
the first $n$ elements of $\mathbb{N}$ follows a power law. A different but related
construction, based on scaling by a variable derived by 
\citet*{Bertoin:Fujita:Roynette:Yor:2006:1},
yields simultanuous power laws on the number and the sizes of subsets.

We begin with a summary of the main results, including a brief review of definitions and 
preliminaries,
in the remainder of this. 
Technical details and specific examples follow
in Sections \ref{sec:construction}--\ref{sec:tv}. All proofs are collected in the appendix.

\subsection{Background}
\label{sec:overview:review}

A \kword{latent feature model} generates a random binary ${n\times\infty}$ matrix $\Z$ and a 
random sequence ${\U=(U_1,U_2,\ldots)}$ with entries in some space $\atomspace$. 
Suppose data ${X_1,\ldots,X_n}$ is observed, say in $\mathbb{R}^d$, 
where each observation $X_i$ is a list of measurements describing an object $i$. Given a suitable family of distributions $P_{\phi}$ on the sample space, a
latent feature model explains such data as
\begin{equation}
  \label{eq:latent:feature:model}
  X_{ij}\sim P_{\phi_{ij}}
  \qquad\text{ where }
  \phi_{ij}:=(\Z_{i1}U_{1},\Z_{i2}U_{2},\ldots)\;.
\end{equation}
The model assumes an unobserved set of properties, 
called \kword{features}; each object $i$ may posses feature $k$ (${\Z_{ik}=1}$), or not possess it (${\Z_{ik}=0}$). 
The variable $\Z_{ik}$ acts as a switch that turns the effect of parameter $U_{j}$ on $X_{ij}$ on 
or off. Throughout, we always require
\begin{equation}
  \label{eq:row:sum:finite}
  \textstyle{\sum_{k}} \Z_{ik}<\infty \qquad\text{for all }i=1,\ldots,n\;.
\end{equation} 

Various forms of this model are widely used in both machine learning and statistics. Here, 
$\Z$ and $\U$ constitute the model parameters; since both are random variables, the model is Bayesian.
Since both have an infinite number of columns,
the model is nonparametric; a parametric model would be obtained by restricting $\Z$ and $\U$
to a fixed, finite number of columns. 
Typically, however, only a finite (if unbounded) number of parameters should explain
any single observation, which is ensured by the constraint \eqref{eq:row:sum:finite}.
Perhaps the most widely used example of a latent feature model assumes
the parameters $U_{j}$ are vectors of scalar effects that combine linearly, and the only additional randomness are additive, independent noise contribution $\varepsilon_{ij}$, in which case \eqref{eq:latent:feature:model} takes the form
${X_{ij}=\sum_{k}\Z_{ik}U_{jk} + \varepsilon_{ij}}$. 
This model and its variants are encountered in statistics as linear factor analyzers, and in machine learning and
as collaborative filtering methods; see e.g.\ \citep{Griffiths:Ghahramani:2011}.

In more probabilistic terms, latent features models are closely related to random partitions 
of the integers: Encode a partition of $\mathbb{N}$ into disjoint blocks
as a binary matrix $\Z$,
where ${\Z_{ik}=1}$ if $i$ is in block $k$ of the partition, and ${\Z_{ik}=0}$ otherwise. 
Since blocks are disjoint, ${\sum_k\Z_{ik}=1}$ for all $i$.
A latent feature model relaxes this row sum
constraint to the finiteness constraint \eqref{eq:row:sum:finite}.
For example, for the subset ${\lbrace 1,\ldots,6\rbrace}$, the two matrices
\vspace{-.05cm}
\begin{center}
\begin{tikzpicture}
  \begin{scope}
    \node at (0,0) {$
      \left(
      \begin{matrix}
        1 & 0 & 0 & 0 \\
        1 & 0 & 0 & 0 \\
        0 & 1 & 0 & 0 \\
        1 & 0 & 0 & 0 \\
        0 & 0 & 1 & 0 \\
        0 & 0 & 0 & 1 
      \end{matrix}\right)
      $};
    \node at (0,1.7) {block \#};
    \node[rotate=90] at (-1.7,0) {object \#};
  \end{scope}
  \begin{scope}[xshift=5.5cm]
    \node at (0,0) {$
      \left(
      \begin{matrix}
        1 & 0 & 0 & 0  \\
        1 & 0 & 0 & 0  \\
        1 & 1 & 0 & 1  \\
        1 & 0 & 1 & 0  \\
        0 & 0 & 1 & 0  \\
        0 & 0 & 0 & 1 
      \end{matrix}\right)
      $};
    \node at (0,1.7) {block \#};
    \node[rotate=90] at (-1.7,0) {object \#};
  \end{scope}
\end{tikzpicture}
\end{center}
respectively represent the partition and family of sets
\begin{equation}
  (\lbrace 1,2,4 \rbrace,
  \lbrace 3 \rbrace,
  \lbrace 5 \rbrace, 
  \lbrace 6 \rbrace
  )\;
  \qquad\text{ and }\qquad
  ( 
  \lbrace 1,2,3,4 \rbrace,
  \lbrace 3 \rbrace,
  \lbrace 4,5\rbrace,
  \lbrace 3,6 \rbrace
  )\;.
\end{equation}
The random matrix $\Z$ in a latent feature model thus encodes a random family of sets.
In the following, we consider two types of representations for the law of $\Z$:
\begin{itemize}
\item \kword{Urn schemes} generate each row of $\Z$ conditionally on the
  previous ones. Following the interpretation of $\Z$ as a generalization of a random
  partition to a family of sets, these urn schemes can be regarded as a generalization
  of the Blackwell-MacQueen urn scheme.
  \vspace{.1cm}
\item \kword{Hierarchical representations} first generate a random discrete measure 
  \begin{equation}
    \label{eq:hierarchical:rmeasure}
    \xi(\argdot)\equdist\sum_{k\in\mathbb{N}}J_k\delta_{U_k}(\argdot) 
    \qquad\text{ where }J_1,J_2,\ldots\in[0,1] \text{ a.s.}
  \end{equation}
  on $\atomspace$.
  The matrix $\Z$ is then generated column-wise as
  \begin{equation}
    \label{eq:Bernoulli:sampling:basic}
    \Z_{1k},\ldots,\Z_{nk}|(J_k)\simiid\Bernoulli(J_k) \;.
  \end{equation}
\end{itemize}
The hierarchical representation generates matrices whose row vectors form an
exchangeable sequence. That is not typically true for urn schemes, since each row is generated
conditionally on the previous one. The encoding of family of sets by $\Z$ is not unique, however,
and we call two matrices \kword{equivalent} if both encode the same family of sets, and
denote the equivalence class of $\Z$ as ${[\Z]}$. We call $\Z_1$ and $\Z_2$
\kword{equivalently distributed} if ${[\Z_1]\equdist [\Z_2]}$.

The urn scheme representation was first used by \citet{Griffiths:Ghahramani:2006},
who introduced a latent feature model which they called the \kword{Indian buffet process} (IBP).
To specify urn schemes, we frequently rely on the quantities
\begin{equation}
  \label{eq:nk:Kn}
  n_k:=\sum_{i=1}^n\Z_{ik}
  \qquad\text{ and }\qquad
  K_n:=\sum_k\indicator{n_k>0}\;,
\end{equation}
\ie in terms of the object/feature interpretation, $n_k$ of the first $n$ observed object possess feature $k$,
and $K_n$ is the total number of distinct feature exhibited by the first $n$ objects.
The following version is due to \citet*{Ghahramani:Griffiths:Sollich:2006:1}:
Choose ${c,\theta>0}$.
\begin{enumerate}
\item The first row of $\Z$ contains ${K_1\sim\Poisson(c)}$ consecutive 
  non-zero entries.
\item In row ${n+1}$, generate the first $K_n$ entries as 
    ${\Z_{n+1,k}\sim\Bernoulli\bigl(\frac{n_k}{\theta+n}\bigr)}$,
  then append ${\Poisson\bigl(\frac{c\theta}{\theta+n}\bigr)}$
  consecutive non-zero entries.
\end{enumerate}
By fixing ${\theta=1}$, one obtains the original, one-parameter IBP in \citep{Griffiths:Ghahramani:2006}.

The hierarchical representation using a random measure 
was introduced as an alternative
construction of the IBP by \citet{Thibeaux:Jordan:2007}, 
who showed that the matrix $\Z$ generated by \eqref{eq:Bernoulli:sampling:basic}
is equivalently distributed to an ${\IBP(c,\theta)}$ matrix if the sequence ${(J_k)}$
is generated as the jumps of a subordinator with \Levy density
\begin{equation}
  \label{eq:beta:process}
  \rho(s)=\frac{c}{\theta}
  s^{-1}(1-s)^{\theta-1}\;.
\end{equation}
A random measure ${\xi=\sum_kJ_k\delta_{U_k}}$ is called \kword{homogeneous} if the
atom locations ${U_1,U_2,\ldots}$ are \iid variables and independent of the sequence $(J_k)$.
If the weight sequence $(J_k)$ of a homogeneous random measure is generated by a 
subordinator, then $\xi$ is in particular a \kword{completely random measure} (CRM)
in the sense of \citet{Kingman:1967}. Completely random measures with 
\Levy density \eqref{eq:beta:stable:process} were introduced in this form by \citet{Hjort:1990},
who called them \kword{beta processes}. They also appear, up to a transformation
${s\mapsto -\log(1-s)}$, in \citep{Ferguson:Phadia:1979}.
\citet{Teh:Goerur:2009} generalize \eqref{eq:beta:stable:process},
with an additional parameter ${\alpha>-\theta}$, as
\begin{equation}
  \label{eq:beta:stable:process}
  \rho(s)ds=\frac{c}{B(\alpha+\theta,1-\alpha)}
  s^{-\alpha-1}(1-s)^{\theta+\alpha-1}ds
 \end{equation}
where $B$ is the beta function. The resulting CRM is called the \kword{stable-beta process}.
Via \eqref{eq:Bernoulli:sampling:basic}, the stable-beta process defines a three-parameter
generalization ${\IBP(c,\theta,\alpha)}$ of the Indian buffet process.
For ${\alpha>0}$, the weights of a stable-beta CRM exhibit a power law; hence, the column
sums of the resulting random matrix also have a power law distribution. Note
\eqref{eq:beta:stable:process} is of the form 
${\rho(s)=\frac{1}{s}f(s)}$, where $f$ is the density of the structural distribution
of a two-parameter Poisson-Dirichlet process \citep[][\S 2.3]{Pitman:2006}.
This identity implies the stick-breaking representation of
\citep{Paisley:Zaas:Ginsburg:Woods:Carin:2010}
and 
\citep{Broderick:Jordan:Pitman:2011}.

\subsection{Desiderata}

In principle, a distribution for $\Z$ can be specified by choosing any
suitable random measure $\xi$ and applying \eqref{eq:Bernoulli:sampling:basic}.
An urn scheme can then be obtained by integrating
out the random measure $\xi$ and conditioning each row of $\Z$ on the previous rows.
For a model to be useful, however, it should satisfy a number of properties:
\begin{itemize}
\item The hierarchical representation should exist, \ie the rows of $\Z$ should
  be rendered conditionally independent by some random measure.
\item The random measure $\xi$ should be tractable, \ie it should be possible to 
  simulate its weights and atoms. The beta and stable-beta process above
  can both be simulated using ``stick-breaking constructions''.
\item The urn scheme should be tractable, \ie it should be possible to sample from
  conditionals of the form ${\Law(\Z_{n+1}|\Z_1,\ldots,\Z_n)}$, where $\Z_i$ denotes
  the $i$th row of $\Z$. Note
  the Griffiths-Ghahramani urn scheme requires only Bernoulli
  and Poisson variables.
\item Ideally, the conditional ${\Law(\xi|\Z_1,\ldots,\Z_n)}$ should also be 
  available. It can be used to derive the urn scheme from 
  \eqref{eq:Bernoulli:sampling:basic}, and may also be of interest in its own right
  as the posterior distribution of $\xi$ in the sense of Bayesian statistics.
\end{itemize}
Although most models will only admit either a hierarchical representation or a tractable
urn scheme, the analogy to random 
partitions suggests some distinguished cases may exist---such as the
one- and two-parameter Poisson-Dirichlet in the partition case---for which 
several or all above properties hold. At present, the 
IBP family seems to be the only known class of models for which that is the case; our objective
in the following is to identify others.

\subsection{Main results}
\label{sec:overview:results}

For use in the hierarchical construction \eqref{eq:Bernoulli:sampling:basic}, 
a random measure must almost surely satisfy
\begin{equation}
  \label{eq:unitary}
  J_1,J_2,\ldots\in [0,1]
  \qquad\text{ and }\qquad
  \xi(\atomspace)<\infty\;,
\end{equation}
where the second condition ensures \eqref{eq:row:sum:finite}.
We call a discrete measure \kword{unitary} if it satisfies \eqref{eq:unitary}.

In the following, we consider a class ${\JOT(\rho,\PAlt)}$ of unitary random
measures parametrized by a continuous \Levy density $\rho$ and a probability measure $\PAlt$, both
defined on $\mathbb{R}_+$. We refer to these measures as \kword{generalized Dickman measures},
for reasons explained in \cref{sec:stable}. 
Let ${\Delta_1>\Delta_2>\ldots}$ be the jumps of subordinator with \Levy
density $\rho$, ordered by decreasing size, and define a random measure $\xi$ as
\begin{equation}
  \label{eq:scaled:HCRM:basic}
  \xi(\argdot):={\textstyle \sum_{k\in\mathbb{N}}}
  J_k\delta_{U_k}(\argdot) \qquad\text{ where } J_k:=\frac{\Delta_{k+1}}{\Delta_1}\;,
\end{equation}
so ${J_k\in[0,1]}$ almost surely. 
We denote the law of $\xi$ generically by $\JOT(\rho)$.

The variables ${\Delta_2,\Delta_2,\ldots}$, with the largest
jump $\Delta_1$ removed, are conditionally independent given $\Delta_1$, so
${\Delta_{2:\infty}|\Delta_1}$ is a subordinator. Conditioning on ${\Delta_1=a}$ does not
change the the \Levy density on $[0,a]$, but truncates it at $a$.
The conditional random measure
\begin{equation}
  \label{eq:xi:a}
  \xi_a:=\xi|(\Delta_1=a)
\end{equation}
is hence unitary and completely random, with \Levy density
\begin{equation}
  \label{eq:rho:scaled}
  \rho_a(s)=a\rho(as)\indicator{s\leq 1}\;.
\end{equation}
A natural further step is to randomize $a$:
Choose any non-negative random variable ${\DeltaAlt\sim P^{\circ}}$ and define
the unitary random measure
\begin{equation}
  \label{eq:scaled:HCRM:mixed}
  \xi^{\circ}
  \;
  :=
  \;
  \xi_{\DeltaAlt}
  \;
  =
  \;
  \xi|(\Delta=\DeltaAlt)
  \;.
\end{equation}
The law of $\xiAlt$ is denoted ${\JOT(\rho,\PAlt)}$. 
Clearly, choosing $\DeltaAlt$ as $\Delta_1$ recovers ${\xiAlt\sim\JOT(\rho)}$.

Scaling subordinators yields a natural stick-breaking representation: Since
the weights of a random measure
${\xiAlt\sim\JOT(\rho,\PAlt)}$ can be expanded as
\begin{equation}
  \label{eq:stable:Jk:Rk}
  J_k
  =
  \frac{\Delta_{k+1}}{\DeltaAlt}
  =
  \frac{\Delta_{k+1}}{\Delta_k}\frac{\Delta_k}{\Delta_{k-1}}\cdots\frac{\Delta_2}{\DeltaAlt}\;,
\end{equation}
$\xiAlt$ can be represented as 
\begin{equation}
  \label{eq:xi:sb}
  \xiAlt\equdist
  \sum_{k=1}^{\infty}\Bigl(\prod_{j=1}^k R_k\Bigr)\delta_{U_k}
  \qquad\text{ where }
  R_k:=\frac{\Delta_{k+1}}{\Delta_k}
  \text{ and }
  R_1:=\frac{\Delta_{2}}{\DeltaAlt}
  \;.
\end{equation}
Following \citet{Kingman:1975}, the density of the largest jump of a subordinator
can always be obtained explicitly as
\begin{equation}
  \label{eq:density:largest:jump}
  f_{\Delta_1}(s)=\rho(s)e^{-\Lambda(s)} 
  \qquad\text{ where }\qquad
  \Lambda(s):=\int_s^{\infty}\rho(u)du\;.
\end{equation}
Whenever the function $\Lambda$ has a computationally tractable inverse, the 
representation above can be made explicit:

\begin{theorem}
  \label{theorem:sb}
  For any random measure
  ${\xiAlt\in\JOT(\rho,\PAlt)}$, define the right-continuous inverse 
  ${{\Lambda}^{\! -}(s):=\inf\lbrace t|\Lambda(t)\geq s\rbrace}$ of 
  $\Lambda$ in \eqref{eq:density:largest:jump}, and let
  ${E_1,E_2,\ldots\simiid\text{\rm Exponential}(1)}$. Then for any ${a>0}$,
  \begin{equation}
    \label{eq:theorem:sb:stable:dist:identity}
    (\Delta_2,\Delta_3,\ldots)\big|\DeltaAlt=a
    \quad\equdist\quad
    (\Ma_2,\Ma_3,\ldots)
  \end{equation}
  where ${\Ma_k:=\Lambda^{\! -}\bigl(\Lambda(a)+E_2+\ldots+E_k\bigr)}$ for every ${k\geq 2}$.
\end{theorem}
One can hence simulate $\xiAlt$ by generating ${\DeltaAlt\sim\PAlt}$ and
\begin{equation}
  \label{eq:xi:sb:explicit}
  \xiAlt|(\DeltaAlt=a) \equdist
  \sum_{k\in\mathbb{N}}
  \Bigl(\prod_{j=1}^k
  \frac{\Ma_{k+1}}{\Ma_{k}}
  \Bigr)  
  \qquad\text{ where }\qquad \Ma_1:=a\;.
\end{equation}
In the stable case discussed in \cref{sec:stable}, $\Lambda$ is indeed invertible, and
the random variables $\Ma_k$ have explicit representations in terms of 
${\Lambda^{\! -}=\Lambda^{-1}}$. Even if $\Lambda$ is not invertible, it is typically still
possible to sample $\Ma_k$.


In terms of the desiderata listed above, a matrix $\Z$ generated from a 
${\JOT(\rho,\PAlt)}$ measure with a sufficiently
regular function $\Lambda$ thus has a hierarchical representation by a
tractable random measure; that leaves the conditional
${\Law(\xi|\Z_1,\ldots,\Z_n)}$ and the urn scheme. 
General results on conditional distributions of $\xi$ and $\Z_{n+1}$
are given in \cref{sec:construction}. 
We mention one elementary (but apparently overlooked) fact that may be
of independent interest in modeling and simulation problems:
If ${\mu=\sum_k W_k\delta_{U_k}}$ is a homogeneous CRM with \Levy density $\rho$,
and $h$ is a positive function, the random measure
\begin{equation}
  \mu_h:=\sum_kB_kW_k\delta_{U_k}
  \qquad\text{ where }\qquad B_k|\mu\sim\Bernoulli(h(B_k))
\end{equation}
is distributed as a homogeneous CRM with \Levy density ${h\rho}$.
Thus, tilting by $h$ can be simulated by thinning; see \cref{lemma:thinning}.

To obtain tractable urn schemes, one has to consider specific models, and
a case of particular interest are \Levy densities of the form 
\begin{equation}
  \label{eq:stable:scaleinvariant}
  \rho(s)=c s^{-1-\alpha}\qquad\text{ for } \alpha\geq 0 \text{ and }c>0\;,
\end{equation}
By construction, the measure $\xiAlt$ is not completely random, since scaling by $\Delta_1$ makes the
variables $J_k$ dependent---unless one chooses $\rho$ as in \eqref{eq:stable:scaleinvariant}
and ${\alpha=0}$. Since \eqref{eq:stable:scaleinvariant} implies 
\begin{equation}
  \label{eq:stable:scaleinvariant:rho_a}
  \rho_a(s)=c a^{-\alpha}s^{-1-\alpha}\;,
\end{equation}
$\xi_a$ does not depend on $a$ for ${\alpha=0}$, and $\xiAlt$ is hence a CRM. 
If this random measure $\xiAlt$ is substituted into 
\eqref{eq:Bernoulli:sampling:basic}, one obtains precisely the 
one-parameter $\IBP(c)$ distribution of \citet{Griffiths:Ghahramani:2006}.
It is interesting to note that the \Levy density is, in this case, the
intensity of a scale-invariant Poisson process; this relationship is discussed
in more detail in \cref{sec:Dickman}. 

For ${\alpha>0}$, \eqref{eq:stable:scaleinvariant} describes a stable
subordinator. The one-parameter IBP can hence be regarded as a limiting
case of the stable for ${\alpha\searrow 0}$.
Stick-breaking and conditional probabilities
for the stable case are covered in \cref{sec:stable}. One implication of
these results is that the stable case admits a simple urn scheme:
If ${\xiAlt\sim\JOT(cs^{-\alpha-1},\PAlt)}$ for ${\alpha>0}$, and ${\xiAlt}$ is substituted for 
$\xi$ in \eqref{eq:Bernoulli:sampling:basic}, the resulting matrix $\Z$ can equivalently be generated as 
follows.
\begin{enumerate}
\item The first row of $\Z$ contains $\Poisson(\frac{(\DeltaAlt)^{-\alpha}\alpha}{1-\alpha})$ consecutive non-zero
  entries. 
\item In the $(n+1)$st row, each of the first $K_n$ entries is non-zero
  with probability $\frac{n_k-\alpha}{n+1-\alpha}$. 
  Additionally, $\Poisson(C_n)$ consecutive non-zero entries are appended.
\end{enumerate}
The random variable $C_n$ is a function of a sample from the conditional distribution of $\DeltaAlt$
given the previous first $n$ rows of $\Z$, and defined in detail in
\cref{result:posterior:stable:alpha:alpha}.

Since latent feature models are related to exchangeable partitions, and both are 
generated by a class of random measures, it is obvious to ask how
these classes of random measures are related. Simply normalizing generalized Dickman measures
does not generally seem to yield interesting random probability measures. 
A different picture
emerges, however, if one first conditions on a total mass ${\xiAlt(\atomspace)=t}$
of ${t\leq 1}$:
\begin{theorem}
  \label{theorem:pk:simplified}
  Let ${\xiAlt\sim\JOT(\rho,\PAlt)}$, for any choice of $\rho$ and $\PAlt$. 
  Condition $\xiAlt$ on its total mass and normalize, defining the measure
  \begin{equation}
    \overline{\eta}:=\frac{\xiAlt|(\xiAlt(\atomspace)=t)}{t}\;.
  \end{equation}
  Whenever ${t\in(0,1]}$, the random probability measure 
  $\eta$ is a Poisson-Kingman measure. Conversely, every Poisson-Kingman
    measure can be obtained in this manner.
\end{theorem}

One implication is that conditioning on a total mass of at most 1 can drastically
simplify the properties of a random measure. This surprising fact was observed by 
\citet*{Arratia:Barbour:Tavare:1999:1} for the scale-invariant Poisson process, 
and is arguably implicit in the work of \citet{Perman:1993:1}.
Our result specializes to that of 
\citet*[][Theorem 3.1]{Arratia:Barbour:Tavare:1999:1} for 
${\rho(s)=cs^{-1}}$, although their proof does not seem to generalize,
and our result is obtained in a very different manner.
The theorem also shows the random measures $\overline{\xiAltAlt}$
in \eqref{eq:pk:mixed:rm} can be regarded as a refinement of Poisson-Kingman
measures, obtained by including the additional mixing variable $\DeltaAlt$.
A more detailed statement, including the distributions of the resulting measures,
is given in \cref{theorem:PK}.

\begin{remark}
We note two aspects of possible interest to applications:
(i) Given a random partition model with desirable properties, 
  \cref{theorem:pk:simplified} identifies the corresponding feature model; 
  \cref{sec:pk:examples} provides examples.
(ii) An explicit coupling between 
  between the matrix $\Z$ and an exchangeable partition can be constructed, by parametrizing both
  by a single random measure. 
  One can hence specify models for problems where part of the
  data constitutes a latent feature problem and another part a clustering problem. 
  Problems of this type are known in machine learning as multi-task learning problems, see \eg 
  \citep{lounici1050van,NIPS2008_3499,4523930}.
\end{remark}

Finally, we consider the distribution of certain functions of $\Z$: The basic
statistics of interest in feature modeling problems are typically
the row sums (number of features of one object), column sums (number of objects exhibiting a 
feature), and the total number of features $K_n$ as defined in \eqref{eq:nk:Kn}.
Of particular interest in applications are models describing heavy-tailed phenomena,
where one or more of the above statistics exhibit a power law.

Column sums following a power law can be obtained relatively easily,
by choosing $\xi$ such that the weight sequence $(J_k)$ is heavy-tailed,
as is the case for the stable-beta process \eqref{eq:beta:stable:process}
and for several models in \cref{sec:stable}.
We do not consider power laws on the row sums within a single matrix,
since---as pointed out in \cite{Teh:Goerur:2009} and 
\cite{Broderick:Jordan:Pitman:2012:1}---this comes
at the price of sacrificing the representation
\eqref{eq:Bernoulli:sampling:basic}: 
The row sums ${\RS_1,\RS_2,\ldots}$ are conditionally \iid given $\xi$, and
\begin{equation}
  \label{eq:janson}
  \abstmeasure\bigl(|\RS_i - \xi(\atomspace)| > \varepsilon\,\big\vert\,\xi\bigr)
  \leq
  \exp\Bigl(-\frac{\varepsilon^2}{2\xi(\atomspace)-\varepsilon/3}\Bigr)
  \qquad\text{ for all }\varepsilon>0\;,
\end{equation}
by standard Bernoulli tail bounds \citep[see][]{Janson:Luczak:Rucinski:2000:1}.
Within a matrix $\Z$ generated by \eqref{eq:Bernoulli:sampling:basic}, the
distribution of row sums is hence not heavy-tailed.

In \cref{sec:BFRY}, we consider models which result in $K_n$ being marginally
heavy-tailed---that is, if ${\Z^{\ind{(1)}},\Z^{\ind{(1)}},\ldots}$ are 
\iid realizations of $\Z$, each with $n$ rows, 
${K_n(\Z^{\ind{(1)}}),K_n(\Z^{\ind{(1)}}),\ldots}$ empirically exhibit a power law.
That is not the case for the IBP, including the stable-beta case.
One way to achieve this behavior is to use a $\JOT(\rho,\PAlt)$ model and choose $\PAlt$
such that ${1/\DeltaAlt}$ is heavy-tailed. Additionally choosing $\rho$ as stable, for
example, yields a model with power law for the column sums and a marginal power law for 
$K_n$.

There is another, perhaps less obvious construction:
If $\rho$ is of the form \eqref{eq:stable:scaleinvariant:rho_a}, the conditional measure
${\xiAlt|(\DeltaAlt=a)}$ has \Levy density
\eqref{eq:stable:scaleinvariant:rho_a}. In the scale invariant case ${\alpha=0}$,
$\xi_a$ does not depend on $a$. In the stable case ${\alpha>0}$, it does, but 
$a$ only acts as a scaling factor on the \Levy density, and
one can hence equivalently sample the weights $(J_k)$ of $\xiAlt$ as
\begin{equation}
  \label{eq:BFRY:random:measure:intro}
  (J_k)\sim\text{Subordinator}(\zeta \rho)\;,
\end{equation}
where ${\zeta:=(\DeltaAlt)^{-\alpha}}$.
Although this representation coincides with the $\JOT$ representation in the 
scale-invariant and the stable
case, it clearly does not for general choice of $\rho$.
In \cref{sec:BFRY}, we show that there is a 
specific choice of the variable $\zeta$ for which
\eqref{eq:BFRY:random:measure:intro} nonetheless yields a simple urn scheme;
moreover, the distributions of $K_n$, for ${n\in\mathbb{N}}$, can be obtained explicitly.
For a stable-beta density, for example, we obtain the following:
Choose parameters ${\sigma\in(0,1)}$ and ${\alpha,\theta>0}$.
\begin{enumerate}
\item The first row contains ${\Poisson(\phi_1 \zeta)}$ 
  consecutive non-zero entries, where ${\zeta:=G/\beta}$, for
  ${G\sim\GammaDist(1-\sigma,1)}$ and ${\beta\sim\Beta(\sigma,1)}$.
\item
  In the $(n+1)$st row, each of the first $K_n$ entries is non-zero with 
  with probability ${\frac{n_k-\alpha}{\theta+n}}$.
  Additionally, append ${H_{n+1}|K_n\sim\Poisson(G_{n+1}F_{n+1})}$
  consecutive non-zero entries, where ${G_{n+1}\sim\GammaDist(1-\sigma+K_n)}$,
  and the random variable $F_{n+1}$ is a function of $K_n$, the parameters, and
  an independent uniform variable.
\end{enumerate}
The variables $\zeta$ and $F_n$, and the coefficients $\phi_k$, are trivial to evaluate;
all are described in detail in \cref{sec:BFRY}, where we also study
the distribution of $K_n$, and note certain
parallels between \eqref{eq:BFRY:random:measure} and the two-parameter Poisson-Dirichlet distribution.

Since the row sums are not generally Poisson if $\xi$ is not completely random, it stands to reason
to ask for a Poisson approximation. In \cref{sec:tv}, we consider a simple total variation bound
on the approximation error, and show how, as a consequence of \cref{theorem:pk:simplified}, 
the small weights of unitary random measure can be related to those of a random probability measure.
The latter generalizes a result of \citet*{Arratia:Barbour:Tavare:1999:1}.

\section{Conditioning on a set of rows}
\label{sec:construction}

Results in this section are largely auxiliary, but requisite for the ensuing
discussion. Suppose the first $n$ rows
${\Z_{1},\ldots,\Z_{n}}$ of a matrix $\Z$ are generated, using either
an urn scheme, or a hierarchical representation with random measure $\xi$. 
We are concerned
with two types of conditional distributions, the laws of the conditional
variables ${\xi|\Z_{1},\ldots,\Z_{n}}$ and
${\Z_{n+1,}|\Z_{1},\ldots,\Z_{n}}$.
In the model \eqref{eq:latent:feature:model},
$\Z$ and $\U$ are used to explain observational data $X$. In the following,
however, it can be useful to interpret the first $n$ rows of $\Z$ itself
as $n$ observations. In the terminology of Bayesian statistics,
${\Law(\xi|\Z_{1},\ldots,\Z_{n})}$ is then the posterior distribution
of $\xi$. Machine learning algorithms, which represent the matrix
$\Z$ as a ``latent'' variable, require a tractable representation
of this posterior. We use it in the following to compute 
the conditional distribution ${\Law(\Z_{n+1,}|\Z_{1},\ldots,\Z_{n})}$
of the ${(n+1)}$st row in an urn scheme.

\subsection{Bernoulli process}

Following \citet{Thibeaux:Jordan:2007}, we use the following joint
encoding for the variables $\Z$ and $\U$, which is particularly
useful for conditioning:
Given a (fixed) unitary measure ${m=\sum_k w_k\delta_{u_k}}$ on $\atomspace$, 
define a random measure
\begin{equation}
  \label{eq:def:BeP}
  \Pi(\argdot):=\sum_{k\in\mathbb{N}} Z_{k}\delta_{u_k}(\argdot)
  \qquad\text{ with } 
  Z_{k}\sim\Bernoulli(w_k) \text{ independently.}
\end{equation}
$\Pi$ is called a \kword{Bernoulli process} with parameter $m$
in \citep{Thibeaux:Jordan:2007}, and denoted $\BeP(m)$ below.
Let ${\mu=\sum_kW_k\delta_{U_k}}$ be a unitary random measure, and sample
\begin{equation}
  \label{eq:BeP:sampling}
  \Pi_1,\ldots,\Pi_n|\xi\simiid\BeP(\xi)\;.
\end{equation}
The non-vanishing binary weights of each random measure $\Pi_i$ can then be interpreted
as the non-zero entries of the $n$th row of a matrix $\Z$, whose distribution is
clearly equivalent to that in \eqref{eq:Bernoulli:sampling:basic}. Define
$\U$ using the atoms of $\mu$ as ${\U:=(U_1,U_2,\ldots)}$. Then 
$\Pi_i$ provides the values of precisely those atoms that correspond to non-zero
entries in row $i$. 

Since \eqref{eq:def:BeP} thins a point process,
it is not hard to believe---and indeed at times tacitly assumed in the literature---that
$\Pi$ should be Poisson random measure if $\mu$ is completely random.
That is indeed the case:
\begin{proposition}
  \label{result:BeP:Poisson}
  Let ${\mu=\sum_{k}J_k\delta_{U_k}}$ 
  be a unitary, homogeneous CRM with ${G:=\Law(U_1)}$.
  Then $\Pi$, defined by ${\Pi|\mu\sim\BP{\mu}}$, is a Poisson random measure with
  ${\mean{\Pi}=\mean{\mu(\atomspace)}G}$.
\end{proposition}

\subsection{Conditioning}
\label{sec:posterior}

Suppose we choose ${\mu=\xiAlt}$ in \eqref{eq:BeP:sampling} and draw a sample consisting
of $n$ random measures, ${\Pi_{1:n}:=(\Pi_1,\ldots,\Pi_n)}$.
To specify the distribution of ${\xiAlt|\Pi_{1:n}}$, we have to keep track of the
number of features in the sample, and of the prevalence of each feature.
Define $n_k$ and $K_n$ as in \eqref{eq:nk:Kn}.
We say that $U_k$ is \kword{observed} in the sample if ${n_k>0}$. 
Since ${\mass{\xiAlt}<\infty}$ almost surely, the total number $K_n$ of distinct
observed atoms is finite almost surely.
It is customary to denote the observed atoms as ${U_1^{\ast},\ldots,U_{K_n}^{\ast}}$.
Additionally conditioning on $\DeltaAlt$ reduces to the CRM case
\citep{Lijoi:Pruenster:2009:1}.

\begin{proposition}
  \label{result:posterior:xiAlt}
  The conditional random measure ${\xi_n^{\ast}:=\xiAlt|\Pi_{1:n}}$ is distributed as
  \begin{equation}
    \label{eq:posterior:rm}
          {\xi}_n^{\ast}
          \equdist
          \xiAlt_n
          +
          \sum_{k=1}^K
          J_{nk}^{\ast}\delta_{U_k^{\ast}}\;.
  \end{equation}
  Conditionally on $\DeltaAlt$, the unitary measure $\xiAlt_n$ is completely random, and
  ${J_{nk}\condind_{\DeltaAlt}\xiAlt_n}$.
  Let $f_{\DeltaAlt}$ denote the density of $\DeltaAlt$ on $\mathbb{R}_+$, and define
  \begin{equation}
    \label{eq:posterior:auxiliary:variables}
    c_a(n,n_k):=\int_0^a s^{n_k}(1-s)^{n-n_k}\rho(as)ds
    \quad\text{ and }\quad
    \psi_n(a):=a\int_0^1(1-(1-s)^n)\rho(as)ds\;.
  \end{equation}
  The conditional random measure ${\xiAlt_n|(\DeltaAlt=a)}$ is given by the \Levy density
  \begin{equation}
    \label{eq:posterior:density:conditional}
    \rho_n^a(s):=a(1-s)^n\rho(as)\indicator{s\leq 1}\;,
  \end{equation}
  and the jumps $J_{nk}^{\ast}$ in \eqref{eq:posterior:rm} by
  \begin{equation}
    \label{eq:posterior:jump:distribution}
    \mathbb{P}(J_{nk}^{\ast}\in ds|\DeltaAlt=a)=\frac{s^{n_k}(1-s)^{n-n_k}\rho(as)}{c_a(n,n_k)}ds\;.
  \end{equation}
  The conditional law of $\DeltaAlt$ is 
  \begin{equation}
    \label{eq:posterior:DeltaAlt}
    \abstmeasure(\DeltaAlt\in da|\Pi_{1:n})
    \propto
    f_{\DeltaAlt}(a)e^{-\psi_n(a)}\Bigl(\prod_{k=1}^{K_n}c_a(n,n_k)\Bigr)da\;.
  \end{equation}
\end{proposition}

If the random measure $\xiAlt$ is marginalized out of the hierarchical model, 
we obtain the predictive distribution of ${\Pi_{n+1}|\Pi_{1:n}}$. 
\begin{proposition}
  \label{result:posterior:predictive}
  Let ${U_1^{\ast},\ldots,U_{K_n}^{\ast}}$ be observed atoms in $\Pi_{1:n}$.
  Under the predictive distribution $\Law(\Pi_{n+1}|\Pi_{1:n})$,
  each previously observed atom $U_k^{\ast}$ has non-zero weight with probability
  \begin{equation}
    \abstmeasure[\Z_{n+1,k}=1|\Pi_{1:n}]=\frac{c_a(n+1,n_k+1)}{c_a(n,n_k)}\;.
  \end{equation}
  Additionally, there are ${\Poisson(q_n)}$ previously unobserved
  atoms with non-zero weights, where 
  \begin{equation}
    q_n:=\int_0^1 s(1-s)^n\rho(s)ds\;,
  \end{equation}
  and the locations of the newly observed atoms are drawn \iid from $G$.
\end{proposition}

\subsection{Tilting and thinning}

The results above show conditioning a random measure with \Levy density $\rho$ on
Bernoulli process observations leads to \Levy densities 
${\rho(s)(1-s)^n}$. Terms of
the form ${(1-s)^t}$, for some scalar $t$, also arise in the \Levy densities of the beta 
stable beta process \eqref{eq:beta:stable:process}.
The following thinning lemma shows that such terms, and more
generally \Levy densities ${\rho(s)h(s)}$ ``tilted'' by some function $h$, 
can be regarded as the outcome of a conditional thinning operation:

\begin{lemma}
  \label{lemma:thinning}
  Let ${\mu=\sum_{k\in\mathbb{N}}J_{k}\delta_{U_{k}}}$ be a homogeneous CRM
  on $\atomspace$, with \Levy density $\rho$ on ${(0,\infty)}$, 
  and let ${h:(0,\infty)\rightarrow[0,1]}$ be measurable.
  If $(B_k)$ are conditionally independent binary variables with conditional law
  ${B_k|J_k\sim\Bernoulli(h(J_k))}$, then 
  \begin{equation}
    \mu_h:=\sum_{k\in\mathbb{N}}B_kJ_{k}\delta_{U_{k}}
  \end{equation}
  is again a homogeneous CRM, with \Levy density
  ${h\cdot\rho}$. If in particular ${\int_{0}^{\infty}h(s)\rho(s)ds=\infty}$, then 
  ${\mu_h(\atomspace)=\infty}$ almost surely.
\end{lemma}

If a procedure for sampling jumps from $\rho$---such as a stick-breaking representation---is 
available, one can hence sample ${\rho(s)(1-s)^n}$ using 
${B_k\sim\Bernoulli((1-J_k)^n)}$. In the $\JOT(\rho,\PAlt)$ case,
this is applicable to ${\xiAlt_n|(\DeltaAlt=a)}$.

\section{The scale-invariant and the stable case}
\label{sec:stable}

We have already noted that both the basic IBP and the CRP are inherently related to 
the subordinator with \Levy density ${\rho(s)=\theta s^{-1}}$ for some ${\theta>0}$.
In this section, we consider more generally the case
\begin{equation}
  \label{eq:sec:stable:density}
  \rho(s)=c s^{-1-\alpha}
  \qquad\text{ for }
  \alpha,c>0\;,
\end{equation}
\ie ${\Delta_1,\Delta_2,\ldots}$ are the jumps of a scale-invariant Poisson process
(if ${\alpha=0}$), or of a stable subordinator (if ${\alpha>0}$).

\subsection{Dickman distributions}
\label{sec:Dickman}

The scale-invariant Poisson process derives its name from the fact that
a Poisson process ${\lbrace X_1,X_2,\ldots\rbrace}$ on $\mathbb{R}_+$ 
satisfies the scale-invariance
\begin{equation}
  \lbrace X_1,X_2,\ldots\rbrace \equdist \lbrace bX_1,bX_2,\ldots\rbrace \equdist \lbrace 1/X_1,1/X_2,\ldots\rbrace
\qquad\text{ for every } b>0
\end{equation}
if and only if ${\rho=c s^{-1}}$ for some ${c>0}$. The scale invariance gives rise to 
unique properties; see \citep[][Chapter 12.8]{Daley:Vere-Jones:2008} or
\citep{Arratia:1998:1} for more details.

Let ${T:=\sum_k X_k}$ be the sum of locations of a Poisson process 
${\lbrace X_1,X_2,\ldots\rbrace}$ on $\mathbb{R}_+$ with an arbitrary \Levy density 
$\rho$. Under mild regularity conditions, $T$ has a probability density $g$; if so, this density satisfies
the integral equation
\begin{equation}
  tg(t)=\int_0^{\infty}g(t-s)s\rho(s)ds\;,
\end{equation}
see \eg \citep{Pitman:2003:1}. 
Now consider the scale-invariant Poisson process. If this process is defined
on all of $\mathbb{R}_+$, its total mass is infinite with positive probability. 
We hence restrict the process to the unit interval, which ensures ${T<\infty}$
almost surely.
Then ${\rho(s)=\frac{c}{s}\indicator{s\leq 1}}$, and 
\begin{equation}
  \label{eq:integral:equation:dickman}
  tg(t)
  =
  c \int_0^{1}g(t-s)ds
  =c\bigl(G(t)-G(t-1)\bigr)\;,
\end{equation}
where $G$ denotes the cdf of $g$.
Differentiating on both sides yields
\begin{equation}
  \label{eq:differential:equation}
  t\frac{dg}{dt}(t)=(c-1)g(t)-c g(t-1)\;.
\end{equation}
Continuous solutions $g$ to \eqref{eq:differential:equation} exist and are uniquely determined
up to scaling. For ${c=1}$, \eqref{eq:differential:equation} is called 
\kword{Dickman's equation}, and the solution $g$ uniquely determined by
${g(1)=1}$ is the \kword{Dickman function} 
\citep[e.g.][]{Arratia:Barbour:Tavare:2003,Penrose:Wade:2004:1}.
We are interested only in continuous solutions which are probability densities. These
are uniquely determined for each ${c>0}$. For each $c$, an $\mathbb{R}_+$-valued
random variable $D_{c}$ with density $g$ is called a \kword{generalized Dickman} variable.
Thus, the total mass $\mass{}$ of a subordinator with \Levy density
${\rho(s)=c s^{-1}}$ has law $\LAW(D_{c})$.


\subsection{Stick-breaking}

For the one-parameter IBP, and hence for the scale-invariant case ${\alpha=0}$, a stick-breaking
representation was proposed by \citet*{Teh:Goerur:Ghahramani:2007}. It is recovered from
\eqref{eq:xi:sb} by choosing ${R_k\sim\Beta(c,1)}$.

\begin{remark}
  \label{remark:sb:TGG}
  Like the model itself, the representation in \citep{Teh:Goerur:Ghahramani:2007}
  can be understood as an outcome of scaling a subordinator:
  If a point process is scale-invariant, its image under a logarithm is translation-invariant.
  If ${\lbrace X_1,X_2,\ldots\rbrace}$ is scale-invariant Poisson with parameter $c$, the process 
  ${\lbrace -\log X_1,-\log X_2,\ldots\rbrace}$ is again Poisson, now with constant rate 
  $c$, so its
  arrival times are distributed as  ${-\log X_k\equdist E_1+\ldots+E_k}$ for independent 
  ${\text{Exponential}(1/c)}$ variables $E_k$
  \citep[see][]{Arratia:1998:1}. If $U$ denotes a $\Uniform[0,1]$ variable, 
  ${\Law(e^{-E_k})=\Law(U^{1/c})=\Beta(c,1)}$.
\end{remark}

Now consider a general ${\JOT(\rho,\PAlt)}$ model.
In the stable case, \cref{theorem:sb} shows that the coefficients 
$\Ma_k$ in \eqref{eq:xi:sb:explicit} are
\begin{equation}
  \Ma_k=\bigl(\frac{\alpha}{c}(E_1+\ldots+E_{k-1})+a^{-\alpha}\bigr)^{-1/\alpha}\;. 
\end{equation}
Another specific example is the beta process density
${\rho(s)=\alpha s^{-1-\alpha}(1-s)^{\alpha-1}}$, for which
\begin{equation}
  \label{eq:theorem:sb:mod}
  \Ma_k=\bigl(({\textstyle \bigl(\frac{1-a}{a}}\bigr)^{\alpha}+E_1+\ldots+E_{k-1})^{1/\alpha}+1\bigr)^{-1}\;.
\end{equation}
The distribution of the variables $R_k$ now depends on the choice of $\PAlt$.
It is interesting to note that the law $\Beta(c,1)$ in the IBP case above precisely
matches that occurring in the stick-breaking construction of the one-parameter
Poisson-Dirichlet process. 
The corresponding laws in the two-parameter Poisson-Dirichlet are
${R_k\sim\Beta(\theta+\alpha k,1)}$. 
Note that, if ${\alpha=0}$, the one-parameter case above is obtained for ${\theta=c}$.
Thus, $\theta$ and $c$ play the same role when comparing the one- and two-parameter
Poisson-Dirichlet to each other, whereas in the three-parameter models considered
here, they are distinct. In a stable $\JOT$ model, jumps of the form
${\Beta(\theta+\alpha k,1)}$ are obtained for a specific choice of $\PAlt$:
\begin{corollary}
  If ${\xiAlt\sim\JOT(\rho(s)=c s^{-1-\alpha},\PAlt)}$ 
  for ${\alpha>0}$, and
  ${\PAlt:=\GammaDist(\frac{\theta+\alpha}{\alpha},\frac{c}{\alpha})}$ 
  with ${\theta>-\alpha}$, each coefficient $R_k$ has law 
  ${\Beta(\theta+\alpha k,1)}$.
\end{corollary}

\subsection{Conditional distributions}
If $\rho$ is scale-invariant, \cref{result:posterior:xiAlt} shows 
the random measure component $\xiAlt_n$ of ${\xiAlt|\Pi_{1:n}}$,
additionally conditioned on ${\DeltaAlt=a}$, has \Levy density
\begin{equation}
  \rho_{n,a}(s)=c s^{-1}(1-s)^n\;,
\end{equation}
and is hence a beta process. In this sense, the family of beta processes
\eqref{eq:beta:process} derives
from the scale-invariant Poisson process, by
taking the ``closure under sampling''.

Now consider the stable case. 
Again by \cref{result:posterior:xiAlt}, 
$\xiAlt_n$ is given by the
\Levy measure
\begin{equation}
  \rho_{n,a}(s)=ca^{-\alpha}s^{-1-\alpha}(1-s)^n\;.
\end{equation}
To obtain the full law of ${\xiAlt|\Pi_{1:n}}$
additionally requires the law of ${\DeltaAlt|\Pi_{1:n}}$.
To simplify expressions, we consider only the special case ${c=\alpha}$, though this restriction
is not vital. The (surrogate) largest jump ${\DeltaAlt\sim\PAlt}$ repeatedly appears
below raised to a negative power, and it is convenient to express results for $\DeltaAlt$
in terms of an equivalent random variable $\zeta$.
\begin{corollary}
  \label{result:posterior:stable:alpha:alpha}
  Let ${\rho(s)=\alpha s^{-1-\alpha}}$ and ${\xiAlt\sim\JOT(\rho,\PAlt)}$, and define
  ${\zeta:=(\DeltaAlt)^{-\alpha}}$. Then
  \begin{equation}
    \zeta|\Pi_{1:n}\equdist(a^{\alpha}\psi_n(a)+1)Y \qquad\text{ where } Y\sim\GammaDist(K_n+1,1)\;,
  \end{equation}
  and ${\psi_n(a)}$ is defined as in \eqref{eq:posterior:auxiliary:variables}.
\end{corollary}
We can then marginalize out $\xi$ and determine the predictive distribution
\citep[see also][]{Teh:Goerur:2009}:
\begin{corollary}
  Choose ${\rho(s)=\alpha s^{-1-\alpha}}$,
  and suppose samples are generated from a Bernoulli
  process $\BeP(\xiAlt)$ parametrized by a random measure ${\xiAlt\sim\JOT(\rho,\PAlt)}$. 
  Define a random variable $\zeta_n$ by ${\zeta_n^{-1/\alpha}\sim\Law(\DeltaAlt|\Pi_{1:n})}$.
  Then
  \begin{equation}
    \Pi_{n+1}|\Pi_1,\ldots,\Pi_n
    \equdist
    \sum_{j=1}^{N(C_n)}\delta_{U_j}+\sum_{k=1}^{K_n}\Z_{n+1,k}\delta_{U_k}\;,
  \end{equation}
  where $N(C_n)$ is a $\Poisson(C_n)$ variable, and
  the variables $C_n$ and $\Z_{n+1,k}$ are conditionally independent given $\DeltaAlt$.
  Their laws are
  \begin{equation*}
    C_n
    \equdist
    \zeta_n\alpha B(1-\alpha,n+1)
    \quad\text{and}\quad
    \Z_{n+1,k}\sim\Bernoulli(J_k)
    \quad\text{with}\quad
    J_k\sim\Beta(n_k-\alpha,n-n_k+1)\;,
  \end{equation*}
  where $B$ is the beta function.
  The initial sample is marginally distributed as ${\Pi_1\equdist\sum_{j=1}^{N(C_0)}\delta_{U_k}}$,
  with ${N(C_0)\sim\Poisson(\frac{\zeta\alpha}{1-\alpha})}$.
\end{corollary}

Finally, we can choose a specific surrogate variable $\DeltaAlt$, or equivalently, a
specific distribution for $\zeta$. Of particular interest are cases in which the
resulting distribution is heavy-tailed \citep{Teh:Goerur:2009,Broderick:Jordan:Pitman:2012:1}. 
A heavy-tailed law can be generated for example as follows:
\begin{corollary}
  Choose $\xiAlt$ as in \cref{result:posterior:stable:alpha:alpha},
  and let $\zeta$ be an $\alpha$-stable random variable with density $f_{\zeta}$. Then
  \begin{equation}
    \label{eq:posterior:zeta:heavy:tailed}
    \abstmeasure(\zeta\in dy|\Pi_{1:n})\propto y^{K_n}f_{\zeta}(y)e^{-y\phi_{n,\alpha}}dy\;,
  \end{equation}
  which is the law of a generalized gamma variable which has been size-biased $K_n$ times.
\end{corollary}
\noindent 
The law \eqref{eq:posterior:zeta:heavy:tailed} is easy to sample, 
since one can show that a size-biased generalized gamma variable
is a mixture of a single generalized gamma variable and gamma variables.
It does not yield a closed-form urn scheme, however, unlike the alternative
construction explained in \cref{sec:BFRY}.


\subsection{Constructing the stable-beta by scaling and thinning}

If $\rho$ is a stable \Levy density, the random measure ${\xiAlt\sim\JOT(\rho)}$ is, conditionally
on ${\Delta_1=a}$, given by the \Levy density ${\rho_a(s)=a^{-\alpha}cs^{-\alpha-1}}$, which partly matches the
stable-beta process \eqref{eq:beta:stable:process}, 
up to a term of the form ${(1-s)^{\nu}}$.
The actual stable-beta can hence be obtained using conditional thinning, by choosing
${h(s):=(1-s)^{\gamma+\alpha-1}}$ in \cref{lemma:thinning}: If
\begin{equation}
  \label{eq:stable:beta:by:thinning}
  \xiAlt=\sum_{k\in\mathbb{N}}J_k\delta_{U_k}=\sum_{k\in\mathbb{N}}\frac{\Delta_{k+1}}{\Delta_1}
  \qquad\text{ and }\qquad
  B_k\sim\Bernoulli((1-J_k)^{\gamma+\alpha-1})\;,
\end{equation}
for ${\gamma>1-\alpha}$, then ${\sum_{k\in\mathbb{N}}B_kJ_k\delta_{U_k}}$
is a stable-beta CRM of the form \eqref{eq:beta:stable:process}. 

Since the function ${h}$ must take values in $[0,1]$ for the construction
to be valid, it does not cover the parameter range ${\theta\in(-\alpha,1-\alpha]}$.
We note the stable-beta can alternatively obtained from a stable using a scaling operation, 
for a wider range of parameters:
The transformation ${s\mapsto s/(s+\tau)}$ turns ${\alpha s^{-1-\alpha}}$
into ${\alpha \tau^{-\alpha}s^{-1-\alpha}(1-s)^{\alpha-1}}$;
thinning with ${(1-s)^{\theta}}$ then yields a stable-beta.

 \section{Coupling scaled and normalized random measures}
\label{sec:PK}

\def\nxiAlt{\overline{\xiAlt}}

Let ${\xiAlt\sim\JOT(\rho,\PAlt)}$ 
be a random measure constructed as in \eqref{eq:scaled:HCRM:basic}.
In the following, we generically denote the total mass of any random measure $\mu$
on a space $\atomspace$ as ${\mass{\mu}:=\mu(\atomspace)}$.
Define a random probability measure ${\nxiAlt:=\xiAlt/\mass{\xiAlt}}$.
It is obvious to ask whether
we can characterize the class of random probability measures so obtained from the class $\JOT$; such a characterization
is given by \cref{theorem:PK}, which relates $\JOT$ to the class of Poisson-Kingman measures.

\subsection{Poisson-Kingman measures}
\label{sec:pk:pk}

We briefly recall the Poisson-Kingman class \citep{Kingman:1975,Pitman:2003:1}:
Let $\mu$ be a random measure satisfying ${\mass{\mu}<\infty}$ a.s., so that the
random probability measure ${\overline{\mu}:=\mu/\mass{\mu}}$ is well-defined.
If $\mu$ is in particular a homogeneous CRM with \Levy density $\rho$, $\overline{\mu}$ is called
a Poisson-Kingman measure with parameter $\rho$, and its law is denoted $\PK(\rho)$.

In general, the random variables $\mass{\mu}$ and $\overline{\mu}$ are stochastically dependent---the
total mass before normalization
can carry significant information on the distribution of weights in a given realization of $\overline{\mu}$.
The only exception is the Dirichlet process: If a $\mu$ is a homogeneous CRM, then $\overline{\mu}$ and
$\mass{\mu}$ are independent if and only if $\mu$ is a gamma process, and $\overline{\mu}$ is hence
Dirichlet.
If $\mass{\mu}$ and $\overline{\mu}$ are dependent, 
conditioning the law $\PK(\rho)$ on ${\mass{\mu}=t}$ yields a family of normalized random measures
additionally  parametrized by $t$, denoted 
\begin{equation}
  \PK(\rho|t):=\Law(\overline{\mu}|\mass{\mu}=t)\;.
\end{equation}
We may now randomize $t$, by choosing a random variable $\TAlt$ with 
law $\QAlt$ on $\mathbb{R}_+$, and define the random probability measure
\begin{equation}
  \overline{\mu}^{\circ}:=\overline{\mu}|(\mass{\mu}=\TAlt)
  \qquad\text{ with law }
  \qquad
  \PK(\rho,\QAlt):=
  \Law(\overline{\mu}^{\circ})=\int_{\mathbb{R}_+}\PK(\rho|t)\QAlt(dt)\;.
\end{equation}
This is the 
general class of \kword{Poisson-Kingman measures} in the sense of \citet{Pitman:2003:1}.
If $\rho$ is the \Levy density of a gamma process, $\PK(\rho)$ is a Dirichlet process.
If $\rho$ is an $\alpha$-stable \Levy density and $\QAlt$ the law of a polynomially tilted
stable variable, then $\PK(\rho)$ is the law of normalized stable process, and
$\PK(\rho,\QAlt)$ that of a Pitman-Yor process \citep{Pitman:2003:1}.

\subsection{Normalizing scaled subordinators}

We begin with a general result, and then obtain important special cases as corollaries.
Start with a random measure ${\xi\sim\JOT(\rho)}$, and posit an arbitrary joint distribution for a surrogate
largest jump $\DeltaAlt$, and a surrogate total mass $\TAlt$, requiring only absolute continuity
${\Law(\DeltaAlt,\TAlt)\ll\Law(\Delta_1,\mass{\xi})}$. In other words, choose any non-negative,
measurable function $\omega$ satisfying ${\mean{\omega(\Delta_1,\mass{\xi})}=1}$, and define $\DeltaAlt$
and $\TAlt$ by
\begin{equation}
  \label{eq:PK:surrogate:density}
  \abstmeasure(\DeltaAlt\in da,\TAlt\in dt)=\omega(a,t)\abstmeasure(\Delta_1\in da,\mass{\xi}\in dt)\;.
\end{equation}
Then define the random measure
\begin{equation}
  \label{eq:pk:mixed:rm}
  \xiAltAlt\;:=\;\xi|(\Delta_1=\DeltaAlt,\mass{\xi}=\TAlt)\;=\;
  \xiAlt|(\mass{\xi}=\TAlt)\;,
\end{equation}
generalizing the definition of \eqref{eq:scaled:HCRM:basic}
by additionally mixing over the total mass. The next theorem is a more detailed
version of \cref{theorem:pk:simplified}, and generalizes a result on scale-invariant
Poisson processes by \citet*{Arratia:Barbour:Tavare:1999:1}.
\begin{theorem}
  \label{theorem:PK}
  Conditionally on ${\mass{\xi}=t}$ for any ${t\in(0,1]}$, the random probability measure $\overline{\xiAltAlt}$ 
  obtained by normalizing the random measure $\xiAltAlt$ defined in \eqref{eq:pk:mixed:rm}
  is a Poisson-Kingman measure. More precisely, if $(Q_k)$ are the ranked weights of 
  $\overline{\xiAltAlt}$, then 
  \begin{equation}
    \label{eq:PK:mixing:measure}
    \Law(Q_{1:\infty}|\TAlt=t)=\PK(\rho,\PAlt_t)
    \qquad\text{ where }\qquad
    \PAlt_t:=\Law(\DeltaAlt t|\TAlt=t)\;.
  \end{equation}
  Additionally conditioning on ${\DeltaAlt=a}$ makes the law independent of
  the choice of $\PAlt$, and 
  \begin{equation}
    \label{eq:PK:a:t}
    \Law(Q_{1:\infty}|\TAlt=t,\DeltaAlt=a)=\PK(\rho|at)
  \end{equation}
  holds for any ${t\in(0,1]}$.
\end{theorem}

\noindent

Two special cases of \cref{theorem:PK} are of particular interest: One is the basic case 
${\xiAltAlt\in\JOT(\rho)}$, where ${\DeltaAlt=\Delta_1}$ and ${\TAlt=\mass{\xi}}$.
The other is a choice of $\PAlt_t$ that lets us obtain
arbitrary $\PK(\rho,\gamma)$ measures. For both cases, we can obtain explicit
forms for $\PAlt_t$. We again write $f_{\mass{\Delta}}$ for the Lebesgue
density of ${\mass{\Delta}:=\sum_{k=1}^{\infty}\Delta_k}$.
\begin{corollary}
  \label{corollary:pk:basic}
  Let ${\xiAltAlt\in\JOT(\rho,\PAlt)}$. Then for any ${t\in(0,1]}$, 
    identity \eqref{eq:PK:mixing:measure} holds with
  \begin{equation}
    \PAlt_t(dz)=\frac{z\rho(z/t)f_{\mass{\Delta}}(z)dz}{\int_0^{\infty} y\rho(y/t)f_{\mass{\Delta}}(y)dy}\;.
  \end{equation}
\end{corollary}

Now suppose we wish to obtain ${\PK(\rho,\gamma)}$, for an arbitrary measure $\gamma$.
For simplicity, require the density 
${h:=d\gamma/d\Law(\mass{\Delta})}$ exists. The measure $\gamma$ then has Lebesgue density
${hf_{\mass{\Delta}}}$. 
\begin{corollary}
  \label{corollary:pk:general}
  Let ${h:[0,\infty)\rightarrow[0,\infty)}$ be measurable, with 
      ${\mean{h(\mass{\Delta})}=1}$.
  If $\DeltaAlt$ and $\TAlt$ are defined by choosing the density in
  \eqref{eq:PK:surrogate:density} as
  \begin{equation}
    \omega(a,t):=\frac{h(at)}{a\rho(a)}
    \qquad\text{ then }\qquad
    \PAlt_t(dz)=h(z)f_{\mass{\Delta}}(z)dz
  \end{equation}
  holds for all ${t\in(0,1]}$. Any Poisson-Kingman model can be obtained in this manner.
\end{corollary}
We emphasize that, even though $\PAlt_t$ is now independent of $t$, it is still necessary to
condition on a value ${t\leq 1}$.
The requirement ${\gamma\ll\Law(\mass{\Delta})}$ keeps expressions simple, but
a closer look at \cref{theorem:PK} shows it is not essential; for example,
$\PAlt_t$ could be a point mass.

\subsection{Examples}
\label{sec:pk:examples}

As examples of \cref{theorem:PK}, we consider the cases where $\rho$ is 
a gamma, scale-invariant, or stable subordinator. The scale-invariant case
clarifies how the IBP relates to the Chinese restaurant process. 
An additional example shows which feature model has an analogous relationship
to the two-parameter Chinese restaurant process.

First suppose $\Delta_{1:\infty}$ are the ranked jumps of
a gamma process with parameter $\theta$, which has \Levy density
\begin{equation}
  \label{eq:levy:density:gamma}
  \rho(s):=\theta s^{-1}e^{-s}\;.
\end{equation}
In this case, 
the \Levy measure $\rho_a$ defined in
\eqref{eq:rho:scaled} for the scaled subordinator 
is still of gamma type---albeit with jumps bounded by 1---so it is not very surprising that 
$\overline{\xi}_a$ is a Dirichlet process:
\begin{proposition}[The gamma case]
  \label{result:example:gamma}
  Let ${\xi\in\JOT(\rho)}$, where $\rho$ is the gamma \Levy density
  \eqref{eq:levy:density:gamma}. Then conditionally on ${\mass{\xi}=t\leq 1}$,
  the normalized measure $\overline{\xi}$ is a Dirichlet process with concentration
  $\theta$, \ie its weights have law
  \begin{equation}
    \Law(Q_{1:\infty}|\mass{\xi}=t)=\PD(0,\theta) \qquad\text{ for any }t\in(0,1]\;.
  \end{equation}
\end{proposition}
Although the gamma process yields a one-parameter Poisson-Dirichlet, and hence a CRP,
it is not suitable for constructing the IBP. Rather, the CRP and IBP can be constructed
jointly from a scale-invariant Poisson \citep[see][Theorem 3.1]{Arratia:Barbour:Tavare:1999:1}.
In this case, $\mass{\xi}$ has the generalized
Dickman distribution ${\LAW(D_{\theta})}$ described in \cref{sec:Dickman}.
\begin{proposition}[The IBP and CRP derived from a single random measure]
  \label{result:example:scale:invariant}
  \mbox{ }\\
  If ${\xi\sim\JOT(\rho)}$ for the scale-invariant Poisson density ${\rho(s)=\theta s^{-1}}$, 
  \begin{equation}
  \Law(Q_{1:\infty}|\mass{\xi}=t) 
  =
  \PD(0,\theta) \qquad\text{ for all }t\in(0,1]\;.
\end{equation}
  The random partition induced by ${\overline{\xi}|(\mass{\xi}\leq 1)}$ is hence the Chinese restaurant process
  $\CRP(\theta)$, and the feature model induced by $\xi$ is the $\IBP(\theta)$.
\end{proposition}

Using a stable subordinator, we can define feature models with Pitman-Yor-like properties. This requires
the general class $\JOT(\rho,\PAlt)$, though; the case $\JOT(\rho)$ for stable $\rho$ is mostly of theoretical
interest:
\begin{proposition}[The stable case]
  \label{result:example:stable}
  If ${\xi\sim\JOT(\rho)}$ for a subordinator with ${\rho(s)=c s^{-1-\alpha}}$, 
  \begin{equation}
  \Law(Q_{1:\infty}|\mass{\xi}=t) 
  =
  \PD(\alpha,\alpha) \qquad\text{ for all }t\in(0,1]\;,
\end{equation}
  and ${\overline{\xi}|\mass{\xi}=t}$ is hence a Pitman-Yor process with parameters $(\alpha,\alpha)$.
\end{proposition}

That the Dirichlet process can be obtained from the unitary random measure underlying the IBP raises
the question which unitary random measure has a similar relationship to the Pitman-Yor process.
Obtaining the Pitman-Yor process requires the general formulation in \cref{corollary:pk:general}.

\begin{proposition}
  \label{result:example:PY}
  Choose an $\alpha$-stable subordinator, ${\rho(s)=cs^{-1-\alpha}}$, and choose
  $\omega$ as in \cref{corollary:pk:general}, with ${h(s):=s^{-\theta}}$. Then conditionally on ${\mass{\xi}\leq 1}$,
  the random probability measure $\overline{\xiAltAlt}$ is a Pitman-Yor process with parameters
  $(\alpha,\theta)$, \ie
  \begin{equation}
    \Law(Q_{1:\infty}|\mass{\xi}=t)=\PD(\alpha,\theta)
    \qquad\text{ for all }
    t\in (0,1]\;.
  \end{equation}
\end{proposition}

\section{Power laws}
\label{sec:BFRY}

\def\BFRY{\text{\rm BFRY}}
\def\PBFRY{\text{\rm Poisson-BFRY}}

This section constructs a model that is not a $\JOT$ model, but does satisfy the
desiderata in \cref{sec:introduction}, and for which the distribution of $K_n$
is a power law and can be obtained explicitly.

\subsection{Motivation: The two-parameter Poisson-Dirichlet}
One way to construct the two-parameter
Poisson-Dirichlet distribution $\PD(\alpha,\theta)$, loosely based on
Proposition 21 in \citep{Pitman:Yor:1997}, is as follows:
Let $(\Delta_k)$ be the ranked jumps of an $\alpha$-stable subordinator, and ${(\tilde{\Delta}_k)}$
their size-biased permutation. Let ${Y:=E/\tilde{\Delta}_1}$, where $E$ is an independent standard exponential variable.
The sequence ${(\tilde{\Delta}_2,\tilde{\Delta}_3,\ldots|Y=y)}$
is then distributed as the jumps of a subordinator with \Levy density 
\begin{equation}
  \label{eq:BFRY:generalized:gamma:density}
  \rho_{\alpha,y}(s)=cs^{-\alpha-1}e^{-y^{\alpha}s}\;,
\end{equation}
which is a generalized gamma process.
Let ${T:=\tilde{\Delta}_2+\tilde{\Delta}_3+\ldots}$ be its total mass, and mix
the normalized process ${(\tilde{\Delta}_k/T)_{k\geq 2}}$ against a gamma distribution:
If ${\zeta\sim\GammaDist(\theta/\alpha,1)}$, for ${\theta>0}$, then
\begin{equation}
  \Bigl(\frac{\tilde{\Delta}_2}{T},\frac{\tilde{\Delta}_3}{T},\ldots\;\Big\vert\;Y=\zeta^{1/\alpha}\Bigr)
  \sim \text{PD}(\alpha,\theta)\;.
\end{equation}


Observe that the variable $\zeta$, when 
substituted into \eqref{eq:BFRY:generalized:gamma:density}, acts
as a random scaling factor \emph{on the \Levy density} of the process. That suggests 
a similar construction for a unitary random measure:
Choose a \Levy density $\rho$ supported on $[0,1]$ and a positive scalar random variable $\zeta$,
and define a homogeneous random measure as 
\begin{equation}
  \label{eq:BFRY:random:measure}
  (J_k)\sim \text{Subordinator}(\zeta\rho)
  \qquad\text{ and }\qquad
  \mu:=\sum_{k\in\mathbb{N}}J_k\delta_{U_k}\;.
\end{equation}
As already noted in \cref{sec:introduction}, $\mu$ is not generally a $\JOT$ measure, unless
$\rho$ is stable or scale-invariant. 

\subsection{A heavy-tailed case with tractable urn scheme}

For specific choice of $\zeta$, \eqref{eq:BFRY:random:measure} yields a
tractable model, with a simple urn scheme and a power law for $K_n$.
Consider a Bessel process
of dimension ${2(1-\sigma)}$, for some ${\sigma\in(0,1)}$, and denote by
$L_t$ the length of its excursion above 0 for which the excursion interval
contains the time $t$. Let $T$ be an independent standard exponential variable,
and define $\zeta$ as the randomization ${\zeta:=L_T}$. This variable
is studied extensively by \citet*{Bertoin:Fujita:Roynette:Yor:2006:1}, who show
it has the remarkably simple density
\begin{equation}
  \abstmeasure(\zeta\in dz)=\frac{\sigma}{\Gamma(1-\sigma)}z^{-\sigma-1}(1-e^{-z})dz\qquad\text{ for }\sigma\in(0,1)\;,
\end{equation}
and is hence of the form
\begin{equation}
  \zeta:=\frac{G}{\beta}\qquad\text{ where } G\sim\GammaDist(1-\sigma,1) \text{ and } \beta\sim\Beta(\sigma,1)\;.
\end{equation}
We refer to the law of this variable as $\BFRY(\sigma)$, following
\citep{Devroye:James:2014:1}. Clearly, $\zeta$ is heavy-tailed, 
with more mass in the tail as $\sigma$ decreases.

Let $\rho$ be a \Levy density on $[0,1]$, and define
a sequence $\phi_k$ and a function $f$ as
\begin{equation}
  \label{eq:BFRY:def:psi:k}
  \phi_k:=\int_0^1 (1-(1-s)^k)\rho(s)ds
  \qquad\text{ and }\qquad
  f(x,a,b):=(xa^b+(1-x)(1+a)^b)^{1/b}\;.
\end{equation}
Specifically, for a stable-beta density ${\rho(s)=\alpha s^{-\alpha-1}(1-s)^{\theta+\alpha-1}}$,
one obtains
\begin{equation*}
  \phi_k=\frac{\alpha\Gamma(\theta+\alpha+k-1)\Gamma(1-\alpha)}{\Gamma(\theta+k)}\;,
\end{equation*}
and the following urn scheme.
Let ${U_1,U_2,\ldots\simiid\Uniform[0,1]}$, and generate $\Z$ as follows: 
\begin{enumerate}
\item The first row has ${H_1\sim\Poisson(\zeta\phi_1)}$
  consecutive non-zero entries, 
  where ${\zeta\sim\BFRY(\sigma)}$.
\item
  In the $(n+1)$st row, each of the first $K_n$ entries 
  is selected independently with probability ${\frac{n_k-\alpha}{\theta+n}}$.
  Additionally, append
  \begin{equation}
    \label{eq:BFRY:sampling:scheme:aux:3}
    H_{n+1}|K_n\sim\Poisson(G_{n+1}\phi_{n+1} f(U_{n},{\textstyle\sum_{j<k}}\phi_j,\sigma-K_n))
  \end{equation}
  consecutive non-zero entries, where ${G_{n+1}\sim\GammaDist(1-\sigma+K_n)}$.
\end{enumerate}
\begin{remark}
  The only aspect specific to the stable-beta in the urn scheme above is the
  probability ${\frac{n_k-\alpha}{\theta+n}}$ in step (2). The urn scheme holds for
  any \Levy density $\rho$ on $[0,1]$, with $\phi_k$ computed as in \eqref{eq:BFRY:def:psi:k},
  if the probabilities ${\frac{n_k-\alpha}{\theta+n}}$ are replaced appropriately.
  As we show below, however, the choice ${\zeta\sim\BFRY(\sigma)}$ is essential: 
  It yields a closed form of $\Law(K_n)$.
\end{remark}

By definition, the variable ${K_1=H_1}$ has a mixed
Poisson-BFRY distribution. The same turns out to hold generally for all variables $K_n$,
and we write ${K\sim\PBFRY(\sigma,\tau)}$ if 
\begin{equation}
  K|\zeta\sim\Poisson(\tau\zeta) \qquad\text{ for }\qquad \zeta\sim\BFRY(\sigma)
  \qquad\text{ and }\qquad
  \tau>0\;.
\end{equation}
The main result regarding the urn scheme is the following:
\begin{proposition}
  \label{proposition:BFRY:matrix}
  Let $\Z$ be generated by the urn scheme above.
  \begin{enumerate}
  \item 
    $\Z$ is equivalently distributed to a random matrix 
    generated from a random measure $\mu$ according to 
    \eqref{eq:Bernoulli:sampling:basic}, where $\mu$ is the
    random measure \eqref{eq:BFRY:random:measure} and ${\zeta\sim\BFRY(\sigma)}$.
  \item 
    For each ${n\in\mathbb{N}}$, the total number $K_n$ of features selected in $\Z$
    has marginal distribution ${K_n\sim\PBFRY(\sigma,\sum_{j\leq n}\phi_j)}$.
    In particular, $K_n$ exhibits a power law.
  \end{enumerate}
\end{proposition}

The proof of \cref{proposition:BFRY:matrix} follows from the properties of the $\PBFRY$ distribution, 
which are characterized by the following two lemmas.
For any ${\sigma\in(0,1)}$ and ${a,b>0}$, define the function
\begin{equation}
  p_{\sigma,a,b}(j,k):=\frac{\Gamma(k+j-\sigma)(a-b)^j}{j!\Gamma(-\sigma)}
  \frac{a^{-\sigma-j}-(1+a)^{\sigma-j}}{b^{\sigma-k}+(1+b)^{\sigma-k}}
\end{equation}
with arguments ${j,k\in\mathbb{N}\cup\lbrace 0\rbrace}$.
A straightforward computation shows:
\begin{lemma}
  The $\PBFRY(\sigma,\tau)$ law has mass function
  ${\abstmeasure(H=j)=p_{\sigma,\tau,0}(j,0)}$.
\end{lemma}
Next, consider a sequence of $\PBFRY(\sigma,\tau_n)$ variables, generated with different
parameters $\tau_n$, but coupled by a single $\BFRY$ variable $\sigma$. Partial sums of such
variables inherit additivity from the Poisson distribution; the benign properties
of the resulting conditionals account for the existence of the closed-form urn scheme above.
\begin{lemma}
  Let ${\psi_1,\psi_2,\ldots}$ be positive scalars with partial sums ${\tau_n:=\sum_{i\leq n}\psi_i}$.
  Consider random variables ${\zeta\sim\BFRY(\sigma)}$ and 
  ${H_1,H_2,\ldots}$ satisfying
  \begin{equation}
    H_n|\zeta\sim\Poisson(\zeta\psi_n)
    \qquad\text{ and }\qquad
    H_n\condind_{\zeta}(H_1,\ldots,H_{n-1})\;,
  \end{equation}
  Then the following holds:
  \begin{enumerate}
  \item The partial sum ${K_n=\sum_{i\leq n}H_i}$ has law 
    ${K_n\sim\PBFRY(\sigma,\tau_n)}$.
  \item For each $n$, 
    ${\Law(\zeta|H_1=h_1,\ldots,H_n=h_n)=\Law(\zeta|K_n=\sum_{i\leq n}h_i)}$, 
    with density
    \begin{equation}
      \abstmeasure(\zeta\in dz|K_n=k)
      =
      \frac{z^{k-\sigma-1}e^{-z\tau_n}(1-e^{-z})}
           {\Gamma(k-\sigma)(\tau_n^{\sigma-k}-(1+\tau_n)^{\sigma-k})}\;.
    \end{equation}
    Moreover, the conditional can be simulated as ${\zeta|(K_n=k)\equdist V_kG}$,
    for two random variables ${G\sim\GammaDist(k+1-\sigma,1)}$ and
    $V_k$ with density
    \begin{equation}
      \abstmeasure(V_k\in dv)=\frac{(k-\sigma)v^{k-\sigma-1}}{\tau_n^{\sigma-k}-(1+\tau_n)^{\sigma-k}}dv
      \qquad\text{ on }\quad 
      [(\tau_n+1)^{-1},\tau_n^{-1}]\;.
    \end{equation}
  \item The variable ${H_{n+1}|K_n}$ has mass function
    ${\abstmeasure(H_{n+1}=j|K_n=k)=p_{\sigma,\tau_{n+1},\tau_n}(j,k)}$.
 \end{enumerate}
\end{lemma}

\section{Distributions of row and column sums}
\label{sec:tv}

We conclude by considering approximations of the row and column sums,
measures in terms of the total variation distance $\dtv$ between probability
measures. 

Let ${\RS_i:=\sum_k \Z_{ik}}$ be the sum of the $i$th row. Since 
${\mass{\xi}}$ if almost surely finite, so is $\RS_i$. 
If $\xi$ is a homogeneous CRM, 
${\RS_i\sim\Poisson(\mean{\mass{\xi}})}$ by \cref{result:BeP:Poisson}.
If $\xi$ is not completely random, one may still approximate it by a 
Poisson variable; the expected error can be related to the first
size-biased weight of $\xi$:
\begin{proposition}
  \label{result:BP:bound}
  Let ${\xi\equdist\sum_kW_k\delta_{U_k}}$ be a unitary random measure, and 
  $\Z$ a random matrix generated from $\xi$ as in \eqref{eq:Bernoulli:sampling:basic}.  
  Suppose $\RS_i$ is approximated by a Poisson variable with mean $\mass{\xi}$,
  and measure approximation error as
  ${E_{\xi}:=\dtv\bigl(\Law(\RS_i\vert\xi) - \Poisson(\mass{\xi})\bigr)}$.
  Then ${E_{\xi}|\xi<\sum_k W_k^2}$, and the expected error is bounded as
    \begin{equation}
      \mean{E_{\xi}} < \mean{\mass{\xi}\tilde{W}_1}\;,
    \end{equation}
    where ${(\tilde{W}_1,\tilde{W}_2,\ldots)}$ is a size-biased permutation of the
    weight sequence $(W_k)$.
\end{proposition}

Now consider the $k$th column. Since the column entries are conditionally \iid 
with mean $W_k$, the behavior of column sums 
reduces to that of the weights $W_k$. The next result provides a 
tool that relates these weights to those of a random probability
measure, along the lines of \cref{theorem:PK}. The bound considers
only those weights not exceeding a given size $\beta$; 
all larger weights are discarded from the total mass, by defining
\begin{equation}
  \tau_{\beta}(x):=\sum_{i\in\mathbb{N}}x_i\indicator{x_i\leq \beta}
  \qquad\text{ for }\beta>0 
\end{equation}
for a point set ${x=\lbrace x_1,x_2,\ldots\rbrace}$ in $\mathbb{R}_+$. 
Comparing the laws of the remaining jumps using total variation distance
then reduces to comparing the scalar variables $\tau_{\beta}$ to its conditional
given the total mass:
\begin{proposition}
  \label{result:PK:sufficiency}
  Let ${X=\lbrace X_1,X_2,\ldots\rbrace}$ be the set of jump heights of a subordinator with 
  with \Levy density $\rho$, and require that its total mass ${T=\sum X_k}$ a density $f_T$ 
  on $\mathbb{R}_+$. Let 
  ${V=\lbrace V_1,V_2,\ldots\rbrace}$ be the weights of a random 
  probability measure with law ${\PK(\rho,\PAlt_t)}$ as in \eqref{eq:PK:mixing:measure},
  for some ${t\in(0,1]}$. Then
  \begin{equation}
    \label{eq:total:variation}
    \dtv(\Law(V\cap[0,\beta]),\Law(X\cap[0,\beta]))
    =
    \dtv(\Law(\tau_{\beta}(X)|T=t),\Law(\tau_{\beta}(X))) 
  \end{equation}
  holds for all ${\beta\in[0,1]}$.
\end{proposition}
The bound \eqref{eq:total:variation} generalizes a result
for the scale-invariant Poisson process obtained in 
\citep{Arratia:Barbour:Tavare:1999:1}.
As \cref{theorem:PK}, conditioning on ${T\leq t=1}$ on the right-hand side of 
\eqref{eq:total:variation} is essential; for arbitrary values of $T$, the result does
not generally hold.

\vspace{1cm}
{\noindent\textbf{Acknowledgments.}
  LFJ was supported in
  part by grant RGC-HKUST 601712 of the \mbox{HKSAR}.
  PO was supported in part by grant FA9550-15-1-0074 of AFOSR.
  YWT's research leading to these results has received funding from the
  European Research Council under the European Union's Seventh Framework
  Programme (FP7/2007-2013) ERC grant agreement no. 617071.
}

\newpage
\appendix

\section*{Proofs}

All results summarized in \cref{sec:introduction} follow from results in later sections, with the
exception of \cref{theorem:sb}.

\begin{proof}[Proof of \cref{theorem:sb}]
  Following \citet[][Appendix 1]{Kingman:1975}, we note $\Delta_1$ has density
  \eqref{eq:density:largest:jump}, so its cdf
  is ${F_{\Delta_1}(s)=e^{-\Lambda(s)}}$, and ${e^{-\Lambda(\Delta_1)}}$ is a uniform variable. Hence,
  ${\Lambda(\Delta_1)\equdist E_1}$, where $E_1$ is a unit-rate exponential.
  Now apply a similar argument conditionally for ${k>1}$: The ranked jumps
  of a subordinator have the Markov property
  \begin{equation}
    \label{eq:proof:stable:sb:aux:1}
    \Delta_k|\Delta_{k-1}=y,\Delta_{k-2},\ldots,\Delta_1
    \;\;\equdist\;\;
    \Delta_k|\Delta_{k-1}=y\;.
  \end{equation}
  As ${\Delta_{2:\infty}|(\Delta_1=y)}$ is a subordinator with \Levy
  density ${\rho(s)\indicator{s\leq y}}$, the conditional density and cdf of its largest jump
  $\Delta_2$ are, respectively,
  \begin{equation}
    \label{eq:proof:stable:sb:aux:2}
    f(s|y)=\rho(s)e^{-(\Lambda(s)-\Lambda(y))}
    \qquad\text{ and }\qquad
    F(s|y)=e^{-(\Lambda(s)-\Lambda(y))}
    \qquad\text{ for }s\leq y\;.
  \end{equation}
  Iterating the argument shows the conditional density and cdf of
  ${\Delta_k|\Delta_{k-1}=y}$ are again given by \eqref{eq:proof:stable:sb:aux:2}.
  By \eqref{eq:proof:stable:sb:aux:1}, the variable ${\Delta_k|\Delta_{k-1}=y,\Delta_{k-2},\ldots,\Delta_1}$
  hence has conditional cdf ${F(s,y)}$.
  The variables ${M_k:=\Lambda^{-1}(E_1+\ldots+E_{k})}$
  therefore have distribution ${(M_k)\equdist(\Delta_k)}$. Conditioning on ${\Delta_1=a}$, and
  hence on ${E_1=\Lambda(a)}$, yields \eqref{eq:theorem:sb:stable:dist:identity}.
\end{proof}

\subsection{Proofs for \cref{sec:construction}}

For the proof of \cref{result:BeP:Poisson}, 
recall the Laplace functional of a random measure $\mu$, for a non-negative 
Borel function $g$, is defined as ${\mean{e^{-\mu(g)}}}$, and $\mu$ is Poisson iff
\begin{equation}
  \label{eq:poisson:laplace}
  \mean{e^{-\mu(g)}}=e^{-(\mean{\mu})(1-e^{-g})}\;,
\end{equation}
see e.g.\ \citep[][Lemma 12.2]{Kallenberg:2001}.
\begin{proof}[Proof of \cref{result:BeP:Poisson}]
  We show that the random measure $\Pi$ satisfies \eqref{eq:poisson:laplace} if $\mu$ is a CRM.
  The definition in \eqref{eq:def:BeP} implies ${\mean{\Pi}=\mean{\mu}}$, and since $\mu$ is a
  homogeneous CRM, ${\mean{\Pi(\argdot)}=\mean{\mu(\atomspace)}G(\argdot)}$ then follows as claimed.

  Conditionally on the weights $W_k$ 
  of ${\mu=\sum_{k\in\mathbb{N}} W_k\delta_{U_k}}$, the variables $Z_k$ in \eqref{eq:def:BeP} are independent
  Bernoulli, each with expectation $W_k$, 
  hence
  \begin{equation}
    \begin{split}
    \mean{e^{-\Pi(g)}|(W_k)}
    =
    \prod_{k} \mean{e^{-Z_kg(X_k)}|(W_k)}
    = 
    \prod_k ((1-W_k)+W_k\mean{e^{-g(X_1)}})\;.
    \end{split}
  \end{equation}
  If we abbreviate ${h(s):=-\log((1-s)+s\mean{e^{-g(X_1)}})}$, we obtain
  \begin{equation}
    \mean{e^{-\Pi(g)}|(W_k)}
    =
    e^{-\sum_k h(W_k)}\;.
  \end{equation}
  The exponent on the right can be written as
  ${\sum h(W_k)=N(h)}$,
  where ${N:=\sum\delta_{W_k}}$.
  The Laplace functional of $\Pi$ at $g$ then coincides with that of $N$ at $h$:
  \begin{equation}
    \mean{e^{-\Pi(g)}}
    =
    \mean{\mean{e^{-\Pi(g)}|(W_k)}}
    =
    \mean{e^{-N(h)}}\;
  \end{equation}
  Since $(W_k)$ are the jumps of a homogeneous CRM, $N$ is a Poisson random measure with expectation 
  ${\mean{N(ds)}=\rho(s)ds}$. By \eqref{eq:poisson:laplace}, its Laplace functional at $h$ is
  \begin{equation}
    \label{aux:proof:BP:poisson:1}
    \mean{e^{-N(h)}}=\mean{e^{-\int (1-e^{-h(s)})\rho(s)ds}}\;.
  \end{equation}
  Since 
  \begin{equation}
    1-e^{-h(s)}=s(1-\mean{e^{-g(U_1)}})= s\int_{\atomspace} 1-e^{-g(u)}G(du)\;,
  \end{equation}
  the integral in the exponent of \eqref{aux:proof:BP:poisson:1} is
  \begin{equation}
    \int_0^{\infty} (1-e^{-h(s)})\rho(s)ds
    =     
    \int_0^{\infty} (s\int_{\atomspace} 1-e^{-g(u)}G(du))\rho(s)ds
    =
    \mean{\mu(\atomspace)}G(1-e^{-g})\;.
  \end{equation}
  Hence, ${\mean{e^{-\Pi(g)}}=\mean{e^{-\mean{\mu(\atomspace)} G(1-e^{-g})}}}$, and $\Pi$ is 
  indeed Poisson by \eqref{eq:poisson:laplace}.
\end{proof}

\begin{proof}[Proof of \cref{result:posterior:xiAlt}]
  By definition, $\xiAlt$ is completely random given $\DeltaAlt$,
  which implies \eqref{eq:posterior:rm}, and the conditional distributions
  of $\xiAlt_n$ and the conditional weights $J^{\ast}_{nk}$ in \eqref{eq:posterior:density:conditional}
  and \eqref{eq:posterior:jump:distribution} can be derived from
  the results of \citet{Kim:1999:1}.
  To obtain the conditional \eqref{eq:posterior:DeltaAlt} of $\DeltaAlt$, we note that
  all information in a sample ${\Pi_1,\ldots,\Pi_n}$ relevant to conditioning
  is summarized by the $K_n$ atom locations $U^{\ast}_k$ and their multiplicities $n_k$.
  The likelihood of observing the sample given ${\DeltaAlt=a}$ is
  \begin{equation}
    \abstmeasure(du_{1:K_n}^{\ast},n_{1:K_n},n|\DeltaAlt=a)
    =
    e^{-\psi_n(a)}\prod_{k=1}^{K_n}c_a(n,n_k)G(du^{\ast}_k)\;.
  \end{equation}
  If $\DeltaAlt$ has density $f_{\DeltaAlt}$, Bayes' theorem hence yields
  \eqref{eq:posterior:DeltaAlt}.
\end{proof}

\begin{proof}[Proof of \cref{lemma:thinning}] 
  Let $G$ denote the law of $U_1$. Since ${(J_{k},B_{k},U_{k})}$ are the points of a 
  marked Poisson process, the Laplace transform 
  ${\mean{{\mbox e}^{-\mu_{h}(f)}}={\mbox e}^{-\psi(f)}}$ of $\mu_h$ is given by
  \begin{equation}
    \begin{split}
    \psi(f)
    =&
    \sum_{a=0}^{1}\int_{\atomspace}\int_{0}^{\infty}(1-{\mbox e}^{-af(y)s})h^{a}(s)[1-h(s)]^{1-a}\rho(s)dsG(du)\\
    =&
    \int_{\atomspace}\int_{0}^{\infty}(1-{\mbox e}^{-f(y)s})h(s)\rho(s)dsG(du)\;,
    \end{split}
  \end{equation}
  since ${(1-{\mbox e}^{-af(y)s})=0}$ for ${a=0}$.
\end{proof}

\subsection{Proofs for \cref{sec:PK}}

This section gives the proof of \cref{theorem:PK}, which we break down into a number of lemmas.
Let $f_{\Delta_1}$ be the density of $\Delta_1$, and define the function
\begin{equation}
  g_{\rho}(a,t)=a\rho(a)f_{\Delta_1}(at)\;.
\end{equation}
We first note the following simple form for the joint density of ${(\Delta_1,\mass{\xi})}$,
which is valid only if $\mass{\xi}$ takes values in $[0,1]$;
outside this range, the expressions become considerably more complicated. 
\begin{lemma}
  \label{lemma:densities}
  Let ${\xi\sim\JOT(\rho)}$ for a \Levy density $\rho$.
  Then ${f_{\Delta_1,\mass{\xi}}(a,t)=g_\rho(a,t)}$ for ${t\leq 1}$.
\end{lemma}

\begin{proof}
  The total mass $\mass{\Delta}$ and largest normalized jump $V$ of the subordinator $(\Delta_1,\Delta_2,\ldots)$ 
  are related to the variables $\Delta_1$ and $\mass{\Delta}$ by
  \begin{equation}
    \label{eq:proof:densities:aux:1}
    V=\frac{\Delta_1}{\mass{\Delta}}=\frac{1}{1+\mass{\xi}}
    \qquad\text{ and }\qquad
    \mass{\Delta}=(1+\mass{\xi})\Delta_1\;.
  \end{equation}
  The hypothesis ${\mass{\xi}\leq 1}$ is hence equivalent to ${V\geq\frac{1}{2}}$. 
  Let ${f_{\mass{\Delta}}}$ be the Lebesgue density of $\mass{\Delta}$.
  By \citep[][Proposition 45]{Pitman:Yor:1997}, the joint density of
  $\mass{\Delta}$ and $V$ is
  ${f_{V,\mass{\Delta}}(a,\tau)=\tau\rho(\tau a)f_{\mass{\Delta}}(\tau-a\tau)}$,
  provided that ${V\geq\frac{1}{2}}$. The change of variables
  ${\tau\mapsto(1+t)a}$ and ${v\mapsto\frac{1}{1+t}}$, and renormalizing
  with respect to $t$, yields
  ${f_{\Delta_1,\mass{\xi}}(a,t)
    =
    {\textstyle\frac{1}{1+t}}f_{V,\mass{\Delta}}\bigl((1+t)a,{\textstyle\frac{1}{1+t}}\bigr)}$,
  which is just ${g_{\rho}(a,t)}$.
\end{proof}

Next, consider the effect of substituting $\rho$ by $\rho_a$ in a Poisson-Kingman partition.

\begin{lemma}
  \label{lemma:PK:at}
  Let $\rho$ be a \Levy measure and ${\rho_a(s)=a\rho(as)\indicator{s\leq 1}}$, for some ${a>0}$.
  Then ${\PK(\rho_a|t)=\PK(\rho|at)}$ whenever ${t\in[0,1]}$.
\end{lemma}

\begin{proof}
  By \citep[][Theorem 4]{Pitman:2003:1}, the EPPF of a $\PK(\nu|t)$ partition for some
  \Levy density $\nu$ is given by
  \begin{equation}
    \label{eq:pitman:PK:EPPF}
    p_{\nu}\bigl(|A_1|,\ldots,|A_k|\big\vert t\bigr)=
    \int_0^t \frac{f(t-s)}{t^nf(t)}s^{n+K-1}
    \int_{\mathcal{S}_k}\Bigl(\prod_{k=1}^K\nu(su_k)u_k^{|A_k|}\Bigr)
    du_{1:K-1}ds\;,
  \end{equation}
  where $\mathcal{S}_K$ denotes the standard simplex in $\mathbb{R}^K$,
  ${du_{1:K-1}=du_1\cdots du_{K-1}}$, and
  $f$ is the density of the total mass of the random measure defined by $\nu$.
  Hence,
  \begin{equation*}
    p_{\rho_a}\bigl(|A_1|,\ldots,|A_k|\big\vert t\bigr)=
    \int_0^{t} \frac{f_{\Delta}(t-s)}{t^nf_{\Delta}(t)}\frac{s^{n+K-1}}{a^K}
    \int_{\mathcal{S}_k}\Bigl(\prod_{k=1}^K\rho(asu_k)\indicator{su\leq 1}u_k^{|A_k|}\Bigr)
    du_{1:K-1}ds\;.
  \end{equation*}
  Since ${s\leq t\leq 1}$ by hypothesis, and ${u_k\leq 1}$, the indicator 
  ${\indicator su \leq 1}$ is always 1, and
  a change of variables ${v(s):=s/a}$ yields
  \begin{equation*}
    p_{\rho_a}\bigl(|A_1|,\ldots,|A_k|\big\vert t\bigr) =
    \int_0^{at} \frac{f_{\Delta}(at-s)}{a^nt^nf_{\Delta}(at)}s^{n+K-1}
    \int_{\mathcal{S}_k}\Bigl(\prod_{k=1}^K\rho(su_k)u_k^{|A_k|}\Bigr)
    du_{1:K-1}ds\;,
  \end{equation*}
  which by \eqref{eq:pitman:PK:EPPF}
  is just ${p_{\rho}\bigl(|A_1|,\ldots,|A_k|\big\vert at\bigr)}$.
\end{proof}

\begin{proof}[Proof of \cref{theorem:PK}]
 Since ${\mass{\xi}\leq 1}$, 
  the conditional density of $\mass{\xi}$ given ${\Delta_1=a}$ is
  \begin{equation}
    f_{\mass{\xi}|\Delta_1}(t|a)
    =
    \frac{g_\rho(a,t)}{\int_0^1 g_\rho(a,s)ds}
    =
    \frac{a\rho(as)f_{\mass{\Delta}}(t)}{\int_0^1 a\rho(as)f_{\mass{\Delta}}(s)ds}
    \qquad\text{ for }t\in[0,1],
  \end{equation}
  by \cref{lemma:densities}. The random measure 
  ${\overline{\xi}|(\Delta_1=a)}$ is hence a Poisson-Kingman measure with
  \Levy density $\rho_a$, and by \cref{lemma:PK:at}, we have
  \begin{equation}
    \Law(\overline{\xi}|\Delta_1=a,\mass{\xi}=t)=\PK(\rho_a|t)=\PK(\rho|at) \qquad\text{ whenever }t\leq 1\;,
  \end{equation}
  which proves \eqref{eq:PK:a:t}.
  To establish \eqref{eq:PK:mixing:measure}, 
  note the conditional density of ${\DeltaAlt|\TAlt}$ satisfies\\
  ${f_{\DeltaAlt|\TAlt}(a|t)\propto \omega(a,t)g(a,t)}$ by \cref{lemma:densities}. Hence,
  \begin{equation*}
    \Law(\overline{\xi}|\TAlt=t)
    =
    \int_{\mathbb{R}_+}\Law(\overline{\xi}|\DeltaAlt=a,\TAlt=t)f_{\DeltaAlt|\TAlt}(a|t)da
    \propto\int_{\mathbb{R}_+}\PK(\rho|at)\omega(a,t)g(a,t)da\;.
  \end{equation*}
  Changing variables to ${z(a)=at}$ yields
  \begin{equation}
    \Law(\overline{\xi}|\TAlt=t)
    \propto\int_{\mathbb{R}_+}\PK(\rho|z)\omega({\textstyle\frac{z}{t}},t)g({\textstyle\frac{z}{t}},t)tdz
    \propto\int_{\mathbb{R}_+}\PK(\rho|z){\textstyle\frac{1}{t}}
    \omega({\textstyle\frac{z}{t}},t)g({\textstyle\frac{z}{t}},t)dz\;.
  \end{equation}
  Since ${\PK(\rho,\PAlt_t)=\int_{\mathbb{R}_+}\PK(\rho|z)\PAlt_t(dz)}$, indeed ${\PAlt_t=\Law(\DeltaAlt t|\TAlt=t)}$
  as claimed.
\end{proof}

The examples in \cref{result:example:gamma}--\cref{result:example:PY}
are obtained as follows:
In the gamma case, in \cref{result:example:gamma}, 
the total mass $\mass{\Delta}$ of the subordinator has 
distribution $\GammaDist(\theta,1)$. For this choice of $\lambda$,
Poisson-Kingman model satisfy ${\PK(\rho|t)=\PK(\rho)=\PD(0,\theta)}$ for all ${t>0}$ 
(see \citep[\S 5.1]{Pitman:2003:1}). Substituting
into \cref{corollary:pk:basic} hence yields
\begin{equation}
\Law(Q_{1:\infty}|\mass{\xi}=t)=\PK(\rho,\PAlt_t)=\PK(\rho)=\PD(0,\theta)
\qquad\text{ for } t\in(0,1]\;.
\end{equation}
Note that, in this case,
${\PAlt_t(dz)\propto z^{\theta-1}e^{-z(1+t)/t}}$.

In the scale-invariant case ${\lambda(s)=\theta s^{-1}}$
in \cref{result:example:scale:invariant}, $\xi|(\Delta_1=a)$ has
\Levy density ${\lambda_a=\theta s^{-1}\indicator{s\leq 1}}$
by \eqref{eq:rho:scaled}, and the total mass
$\mass{\xi}$ is a Dickman variable $D_{\theta}$
as in \cref{sec:Dickman}.
\cref{theorem:PK} shows
${\Law(Q_{1:\infty}|\mass{\xi}=t,\DeltaAlt=a)=\PK(\rho|D_{\theta}=at)}$.
To evaluate ${\PK(\rho|D_{\theta}=at)}$, we reduce to a gamma \Levy density
${\rho'(s)=\theta s^{-1}e^{-s}}$, using the invariance 
\begin{equation}
  \label{eq:PK:exp:tilt}
  \PK(e^{-as}\rho|t)=\PK(\rho|t)\qquad\text{ for all }t>0
\end{equation}
of $\PK$ measures under exponential tilting \citep[\S 4.2]{Pitman:2003:1}.
Since exponentially tilting $\rho_a$ yields 
\begin{equation}
e^{-as}\rho_a(s)=\theta s^{-1}e^{-as}\indicator{s\leq 1}=\rho'_a(s)\;,
\end{equation}
we obtain ${\PK(\rho|D_{\theta}=at)=\PK(\rho'|\mass{\xi}=at)=\PD(0,\theta)}$.

In \cref{result:example:stable}, the subordinator is $\alpha$-stable, hence 
${\rho(s)=cs^{-1-\alpha}}$ and ${T_{\Delta}\sim\Stable(\alpha)}$. 
By \cref{corollary:pk:basic}, 
${  \PAlt_t(dz)\propto z^{-\alpha-1}f_{\mass{\Delta}}(z)dz}$,
which is the law of a polynomially $\alpha$-tilted, $\alpha$-stable variable $S_{\alpha,\alpha}$.
Hence,
\begin{equation}
    \Law(Q_{1:\infty}|\mass{\xi}=t)=\PK(\rho,\Law(S_{\alpha,\alpha}))\;,
\end{equation}
which is the $\PD(\alpha,\alpha)$ distribution, see \citep{Pitman:2003:1}.

In \cref{result:example:PY}, $\rho$ is an $\alpha$-stable subordinator and $\mass{\Delta}$
is hence an $\alpha$-stable random variable with density $f_{\alpha}$.
By \cref{corollary:pk:general}, 
\begin{equation}
  \Law(Q_{1:\infty}|\mass{\xi}=t)=\PK(\rho,z^{-\theta}f_{\alpha}(z))\;,
\end{equation}
which is the $\PD(\alpha,\theta)$ distribution \citep{Pitman:2003:1}.

\subsection{Proofs for \cref{sec:tv}}
The rows sums follow a generalized binomial distribution, the Poisson approximation error 
can be bounded using Le Cam's inequality \citep{LeCam:1960:1}:
\begin{proof}[Proof of \cref{result:BP:bound}]
  By Le Cam's inequality, the law of $\RS_i$ is approximately
  Poisson, with approximation error ${\sum_k W_k^2}$. 
  Since ${\mass{\xi}<\infty}$ a.s. by hypothesis,
  $\xi$ can be normalized, and we write ${\tilde{W}_1}$ for the first size-biased pick from $W_{1:\infty}$.
  The variable ${\tilde{V}_1:=\tilde{W}_1/\mass{\xi}}$ is hence the 
  first size-biased weight of a random probability measure, and satisfies 
  the identity
  ${\mean{g(\tilde{V}_1)}=\mean{{\textstyle\sum_k} \tilde{V}_k g(\tilde{V}_k)}}$
  for any measurable ${g:(0,\infty)\rightarrow(0,\infty)}$, see \citep[][Eq. (2.22)]{Pitman:2006}. 
  Choosing ${g(v)=t^2v}$, 
  \begin{equation}
    \Bigmean{{\textstyle \sum_k} W_k^2}
    =
    \mathbb{E}_{\mass{\xi}}\Bigl[\Bigmean{{\textstyle \sum_k}t^2\tilde{V}_k^2\Big|\mass{\xi}=t}\Bigr]
    =
    \mathbb{E}_{\mass{\xi}}\bigl[\mean{t^2\tilde{V}_1|\mass{\xi}=t}\bigr]
    =
    \mean{\mass{\xi}^2\tilde{V}_1}
  \end{equation}
  as claimed.
\end{proof}

The next proof uses the following auxiliary result, which
paraphrases \citep[][Corollary 4.1]{Arratia:Barbour:Tavare:1999:1}:
\begin{lemma}
  Let $P$ and $P'$ be two probability measures on $\xspace$ with ${P\ll P'}$, 
  and let ${\phi:\xspace\rightarrow\yspace}$
  be a measurable mapping into a space $\yspace$. Suppose $\phi$ is sufficient for $P$ and $P'$;
  that is, there is some function $f$ such that ${dP/dP'(x)=f(\phi(x))}$. 
  Then the total variation distance between the image measures $\phi(P)$ and $\phi(P')$ satisfies
  ${\dtv(P,P')=\dtv(\phi(P),\phi(P'))}$.
\end{lemma}

\begin{proof}[Proof of \cref{result:PK:sufficiency}]
  Abbreviate ${X_{\beta}:=X\cap[0,\beta]}$.
  From \cref{theorem:PK}, we may substitute ${\Law(X_{\beta}|T=t)}$ for
  $\Law(V)$ in \eqref{eq:total:variation}.
  By \citep[][Corollary 4.1]{Arratia:Barbour:Tavare:1999:1}, we hence have 
  to show that $\tau_{\beta}$ is a sufficient statistic for the pair 
  $\Law(X_{\beta})$ and $\Law(X_{\beta}|T=t)$.
  Let $L$ be Lebesgue measure on $\mathbb{R}_+$.
  Since $T$ has a density ${f_T=d\Law(T)/d L}$, 
  the conditional density ${f_{T|X_{\beta}}=d\Law(T|X_{\beta})/d L}$
  exists. As the continuous-time process ${(T_{\beta})_{\beta\in\mathbb{R}_+}}$ is Markovian,
  $T$ and $X_{\beta}$ are conditionally independent given $T_{\beta}$, and 
  $f_{T|X_{\beta}}$ decomposes as ${f_{T|X_{\beta}}(t|x)=\tilde{f}(t|\tau_{\beta}(x))}$
  for a suitable conditional density function $\tilde{f}$. 
  The joint law then has density
  \begin{equation}
    \frac{\abstmeasure(X_{\beta}\in dx,T\in dt)}{\abstmeasure(X_{\beta}\in dx)\otimes\lambda(dt)}
    =
    f_{T|X_{\beta}}(t|x)\mathbf{1}(x)\;,
  \end{equation}
  where $\mathbf{1}$ is the constant function with value 1. This means the density
  \begin{equation}
    \frac{\abstmeasure(X_{\beta}\in dx|T=t)}{\abstmeasure(X_{\beta}\in dx)}
    =
    \frac{f_{T|X_{\beta}}(t|x)}{f_T(t)}
    =
    \frac{\tilde{f}(t|\tau_{\beta}(x))}{f_{T}(t)}\;.
  \end{equation}
  depends on $X$ only through $\tau_{\beta}$, 
  and $\tau_{\beta}$ is indeed sufficient for
  ${(\Law(X|T),\Law(X))}$. 
\end{proof}

\newpage
\bibliography{JOY_references.bib}
\bibliographystyle{natbib}

\end{document}